\documentclass[a4paper]{amsart}

\usepackage{amsmath, amsfonts, amssymb}
\usepackage{enumerate, tensor, dsfont, mathtools}
\usepackage{enumitem}
\usepackage{graphicx}
\usepackage{tikz}
\usetikzlibrary{arrows, automata}

\usepackage{float}
\usepackage{verbatim}
\usepackage{color}

\numberwithin{equation}{section}

\newtheorem{theo}{Theorem}[section]
\newtheorem{lma}[theo]{Lemma}
\newtheorem{cor}[theo]{Corollary}
\newtheorem{defn}[theo]{Definition}

\newtheorem{prop}[theo]{Proposition}

\newtheorem{exam}[theo]{Example}

\DeclareMathOperator{\R}{\mathbb{R}}

\DeclareMathOperator{\N}{\mathbb{N}}
\DeclareMathOperator{\E}{\mathbb{E}}

\DeclareMathOperator{\Haus}{\mathcal{H}}
\DeclareMathOperator{\Prob}{\mathbb{P}}

\DeclareMathOperator{\geo}{geo}

\DeclareMathOperator{\interior}{int}

\newcommand{\vertiii}[1]{{\left\vert\kern-0.25ex\left\vert\kern-0.25ex\left\vert #1 
    \right\vert\kern-0.25ex\right\vert\kern-0.25ex\right\vert}}
\allowdisplaybreaks

\renewcommand{\epsilon}{\varepsilon}

\title[Dimensions of random self-similar graph directed attractors]{On the dimensions of attractors of random self-similar graph directed iterated function systems}

\author{Sascha Troscheit}
\thanks{The author was financially supported by EPSRC Doctoral Training Grant EP/K503162/1. The author thanks Nikita Sidorov for pointing out a potential link between Assouad dimension and the joint spectral radius, Kenneth Falconer and Mike Todd for many helpful discussions}
\address{Department of Pure Mathematics\\University of Waterloo\\200 University Ave W\\Waterloo\\ON\\N2L 3G1\\Canada}
\email{stroscheit@uwaterloo.ca}

\begin{document}
\begin{abstract}
In this paper we propose a new model of random graph directed fractals that encompasses the current well-known model of random graph directed iterated function systems, $V$-variable attractors, and fractal and Mandelbrot percolation. We study its dimensional properties for similarities with and without overlaps. In particular we show that for the two classes of $1$-variable and $\infty$-variable random graph directed attractors we introduce, the Hausdorff and upper box counting dimension coincide almost surely, irrespective of overlap. Under the additional assumption of the uniform strong separation condition we give an expression for the almost sure Hausdorff  and Assouad dimension.
\end{abstract}

\maketitle

\section{Introduction}\label{introSect}
The study of deterministic and random fractal geometry has seen a lot of interest over the past 30 years. 
While we assume the reader is familiar with standard works on the subject (e.g.~\cite{superfractals}, \cite{FractalGeo}, \cite{TecFracGeo}) we repeat some of the material here for completeness, enabling us to set the scene for how our model fits in with and also differs from previously considered models.

In the study of strange attractors in dynamical systems and in fractal geometry, one of the most commonly encountered families of attractors is the invariant set under a finite family of contractions.
This is the family of \emph{Iterated Function System} (IFS) attractors.
An IFS is a set of mappings $\mathbb I=\{f_i\}_{i\in\mathcal I}$, with associated attractor $F$ that satisfies
\begin{equation}\label{IFSInvariance}
F=\bigcup_{i\in\mathcal I}f_i (F).
\end{equation}
If $\mathcal I$ is a finite index set and each $f_i:\R^d\to\R^d$ is a contraction, then there exists a unique compact and non-empty set $F$ in the family of compact subsets $\mathcal K(\R^d)$ that satisfies this invariance (see Hutchinson~\cite{Hutchinson81}).
These assumptions are however still insufficient to give concrete and meaningful dimensional results for IFS attractors and further assumptions on these maps are stipulated. 
If one considers only similitudes, i.e.\ $\lvert f(y)-f(x)\rvert=c_i\lvert y-x\rvert$, where $c_i\in(0,1)$ is the Lipschitz constant (contraction rate) of $f_{i}$, the attractors are called \emph{self-similar sets}. 
Of particular interest are the dimensional properties of these attractors, with the Hausdorff, packing, and upper and lower box counting dimension being the main candidates for investigation. 
Here we also consider the Assouad dimension, a dimension that was first developed by Assouad~\cite{assouadphd}, \cite{Assouad79} to study embedding problems which has recently gained more traction as a tool to investigate deterministic fractals (see for example \cite{Fraser15a}, \cite{Fraser14a}, \cite{Mackay11}, \cite{Olsen11a} and the references therein). 
One interesting result to note is that for self-similar sets we always have Hausdorff dimension equal to the upper box counting dimension, and therefore the Hausdorff, packing and box counting dimensions coincide (see Falconer~\cite{Falconer89}). 
If one assumes further that the attractors have minimal overlap, that is they satisfy the Open Set Condition, the Hausdorff and box counting dimension coincide with the Assouad and the similarity dimension. The similarity dimension is the unique $s\in\R^+_0$ satisfying the Hutchinson-Moran formula
\begin{equation}\label{HutchinsonMoranFormula}
\sum_{i\in\mathcal I}c_i^s=1,
\end{equation}
(see~\cite{Hutchinson81}, \cite{Moran46}).
 In fact the OSC is not the weakest condition that implies coincidence of Hausdorff and Assouad dimension. 
 The appropriate separation condition here is the \emph{Weak Separation Property} (WSP) (see Fraser et al.~\cite{Fraser15a}).

Graph directed systems are a natural extension of the Iterated Function System (IFS) construction that simultaneously describes a finite collection of sets.
Given a directed multi-graph $\Gamma=(V,E)$ with finitely many vertices $V$ and edges $E$ we consider the collection of sets $\{K_i\}_{i\in V}$.
Let $\tensor*[_v]{E}{_w}$ be the set of edges from $v$ to $w$, we associate a mapping with every edge and the sets $K_{i}$ are described by an invariance similar to (\ref{IFSInvariance}):
\[
K_i = \bigcup_{j\in V}\bigcup_{e\in {}_i E_j}f_e(K_j)\;\;\text{ for all }i\in V.
\]
Assuming the maps $f_{e}$ are contractions, the sets $K_{v}$ are compact and uniquely determined by the graph directed iterated function system.
Note that IFS constructions are also graph directed constructions as these can be modelled by a graph with a single vertex and an edge for every map in the IFS. 
It can also be shown that there exist graph directed attractors that cannot be the attractors of standard IFS, see Boore and Falconer~\cite{Boore13}.
If one further assumes that $\Gamma$ is strongly connected, the Hausdorff, packing, and box-counting dimensions coincide for every attractor $K_i$ and further that all of these notions of dimension coincide.

All of these models have random analogues, which for standard IFS are either the $V$-variable or the $\infty$-variable construction for $V\in\N$. 
Here we will not state the definition of $V$-variable attractors for $1<V<\infty$ and we refer the reader to the seminal papers by Barnsley, Hutchinson and Stenflo~\cite{Barnsley05}, \cite{Barnsley08}, \cite{Barnsley12}.

To explain the construction of a \emph{Random Iterated Function System} (RIFS) one first has to note that the invariant set in (\ref{IFSInvariance}) can also be obtained by iteration of the maps of the IFS. 
Consider the IFS $\mathbb I$ as a self-map on compact subsets of $\mathbb R^{d}$, $\mathbb I:\mathcal K(\R^{d})\to\mathcal K(\R^{d})$, with $X\mapsto \bigcup_{i\in\mathcal I}f_{i}(X)$. 
Take a sufficiently large set $\Delta\in \mathcal K(\R^{d})$, such that $F\subseteq \Delta$, then $F$ can be written as 
\[
F=\lim_{N\to\infty}\bigcap_{k=1}^{N}\mathbb I^{(k)}(\Delta).
\]
For the random analogues of this construction we take a finite collection of Iterated Functions Systems $\mathbb L=\{\mathbb I_{i}\}_{i\in\mathcal L}$ with (finite) index set $\mathcal L$. 
We take a probability vector $\vec\pi=\{\pi_{i}\}_{i\in\mathcal L}$ and consider two random constructions; the 1-variable  \emph{Random Iterated Function System} (RIFS) and the $\infty$-variable Random Iterated Function System.

A $1$-variable RIFS is the limit set one obtains by choosing the IFS that is applied at the $k$-th stage according to probability vector ${\vec\pi}$. This choice of IFS is uniform for that level and the attractor can be written as 
\[
F(\omega)=\lim_{N\to\infty}\bigcap_{k=1}^{N}\mathbb I_{\omega_{1}}\circ\mathbb I_{\omega_{2}}\circ\dots\circ\mathbb I_{\omega_{N}}(\Delta)
\]
with $\omega=(\omega_{1}, \omega_{2},\dots)$, $\omega_{i}\in\mathcal L$ being the infinite sequence chosen according to $\vec\pi$.

The description for $\infty$-variable RIFSs (sometimes called random recursive constructions) differs in the non-uniform application of the same IFS at every level. 
In general they differ substantially, with $V$-variable fractals interpolating between the two. 
We avoid giving a detailed mathematical description at this stage and only comment that an $\infty$-variable attractor is constructed in a recursive manner by assigning a randomly chosen IFS to every finite word that was already constructed, independent of other words and the level of construction, as opposed to a single chosen IFS for every word in the same level of construction. 
This means that every finite word on every level has an independent, but identical in distribution, sequence of IFS maps applied to it.

\begin{figure}[htb]
\begin{tikzpicture}[x=0.85cm,y=0.9cm]

\draw[thick] (0,0)--(14,0);

\draw[thick] (0,-1)--(3.5,-1);
\draw[thick] (5.25,-1)--(8.75,-1);
\draw[thick] (10.5,-1)--(14,-1);

\draw[thick] (0,-2)--(1.167,-2);
\draw[thick] (2.333,-2)--(3.5,-2);
\draw[thick] (5.25,-2)--(6.417,-2);
\draw[thick] (7.583,-2)--(8.75,-2);
\draw[thick] (10.5,-2)--(11.67,-2);
\draw[thick] (12.83,-2)--(14,-2);

\draw[thick] ( 0 ,-3.5)--(0.389,-3.5);
\draw[thick] (0.778  ,-3.5)--( 1.167 ,-3.5);
\draw[thick] (2.333  ,-3.5)--( 2.722 ,-3.5);
\draw[thick] ( 3.111 ,-3.5)--( 3.5 ,-3.5);
\draw[thick] ( 5.25 ,-3.5)--( 5.639 ,-3.5);
\draw[thick] (6.028,-3.5)--( 6.417 ,-3.5);
\draw[thick] (7.5833,-3.5)--( 7.9722 ,-3.5);
\draw[thick] (8.361,-3.5)--( 8.75 ,-3.5);
\draw[thick] (10.5,-3.5)--( 10.889 ,-3.5);
\draw[thick] (11.278,-3.5)--( 11.667 ,-3.5);
\draw[thick] (12.833,-3.5)--( 13.22 ,-3.5);
\draw[thick] (13.611,-3.5)--( 14 ,-3.5);

\node[] at (7,-0.3) {$\epsilon_{0}$};

\node[] at (1.75,-1.3) {$a_{1}$};
\node[] at (7,-1.3) {$a_{2}$};
\node[] at (12.25,-1.3) {$a_{3}$};

\node[] at (.5833,-2.3) {$a_{1}A_{1}$};
\node[] at (2.917,-2.3) {$a_{1}A_{2}$};
\node[] at (5.833,-2.3) {$a_{2}A_{1}$};
\node[] at (8.167,-2.3) {$a_{2}A_{2}$};
\node[] at (11.083,-2.3) {$a_{3}A_{1}$};
\node[] at (13.413,-2.3) {$a_{3}A_{2}$};

\node[] at (14, -0.0) (lvl1) {};
\node[] at (14, -1.) (lvl2) {};
\node[] at (14, -2.) (lvl3) {};
\node[] at (14, -3.5) (lvl4) {};

\path[->, thick, dashed]
(lvl1) edge [bend left] node [right] {$\mathbb I_{2}$}(lvl2)
(lvl2) edge [bend left] node [right] {$\mathbb I_{1}$}(lvl3)
(lvl3) edge [bend left] node [right] {$\mathbb I_{1}$}(lvl4);

\node[] at (7,-2.7) {$\vdots$};
\node[] at (7,-3.8) {$\vdots$};

\end{tikzpicture}
\caption{Generation of a $1$-variable Cantor set by the Iterated Function Systems $\mathbb I_{1}$ and $\mathbb I_{2}$. For each level the IFS is independently chosen and applied uniformly to all codings at that level.}
\label{1vsInfiniteVar1}
\end{figure}
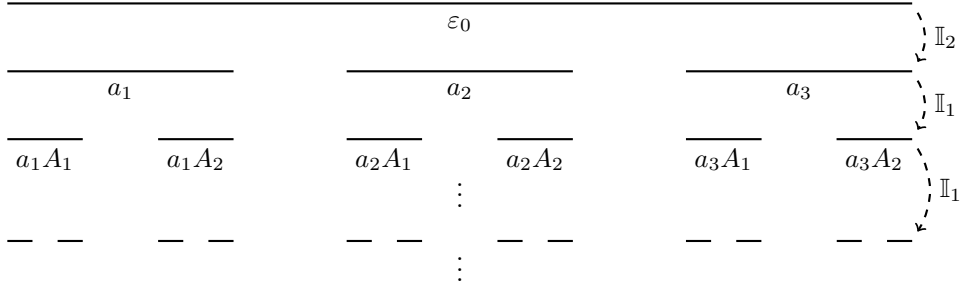

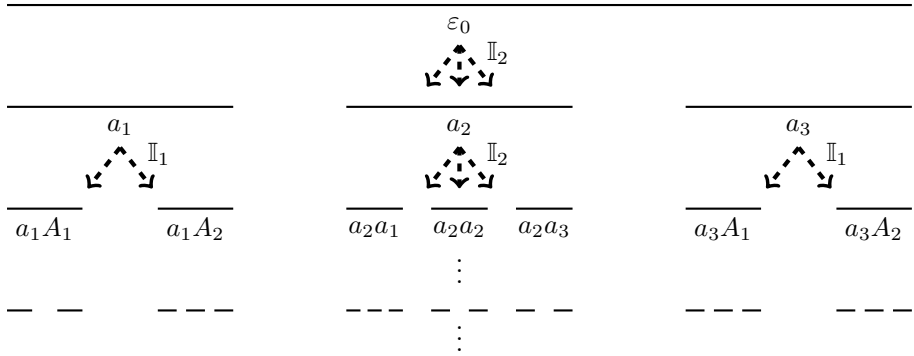
\begin{figure}[htpb]
\begin{tikzpicture}[x=0.85cm,y=0.9cm]
\draw[thick] (0,0)--(14,0);

\draw[thick] (0,-1.5)--(3.5,-1.5);
\draw[thick] (5.25,-1.5)--(8.75,-1.5);
\draw[thick] (10.5,-1.5)--(14,-1.5);

\draw[thick] (0,-3)--(1.167,-3);%
\draw[thick] (2.333,-3)--(3.5,-3);%
\draw[thick] (5.25,-3)--(6.125,-3);%
\draw[thick] (6.5625,-3)--(7.4375,-3);%
\draw[thick] (7.875,-3)--(8.75,-3);%
\draw[thick] (10.5,-3)--(11.667,-3);%
\draw[thick] (12.833,-3)--(14,-3);

\draw[thick] ( 0 ,-4.5)--(0.389,-4.5);
\draw[thick] (0.778  ,-4.5)--( 1.167 ,-4.5);
\draw[thick] (2.333,-4.5)--(2.62475,-4.5);
\draw[thick] (2.7706,-4.5)--(3.06238,-4.5);
\draw[thick] (3.2083,-4.5)--(3.5,-4.5);
\draw[thick] (5.25,-4.5)--(5.46875,-4.5);
\draw[thick] (5.57813,-4.5)--(5.79688,-4.5);
\draw[thick] (5.90625,-4.5)--(6.125,-4.5);
\draw[thick] (6.5625,-4.5)--(6.85417,-4.5);
\draw[thick] (7.14583,-4.5)--(7.4375,-4.5);
\draw[thick] (7.875,-4.5)--(8.16667,-4.5);
\draw[thick] (8.45833,-4.5)--(8.75,-4.5);
\draw[thick] (10.5,-4.5)--(10.7918,-4.5);
\draw[thick] (10.9376,-4.5)--(11.2294,-4.5);
\draw[thick] (11.3753,-4.5)--( 11.667,-4.5);
\draw[thick] (12.833,-4.5)--(13.1248,-4.5);
\draw[thick] (13.2706,-4.5)--(13.5624,-4.5);
\draw[thick] (13.7083,-4.5)--(14,-4.5);

\node[] at (7,-0.3) {$\epsilon_{0}$};

\node[] at (1.75,-1.8) {$a_{1}$};
\node[] at (7,-1.8) {$a_{2}$};
\node[] at (12.25,-1.8) {$a_{3}$};

\node[] at (.5833,-3.3) {$a_{1}A_{1}$};
\node[] at (2.917,-3.3) {$a_{1}A_{2}$};
\node[] at (5.6875,-3.3) {$a_{2}a_{1}$};
\node[] at (7,-3.3) {$a_{2}a_{2}$};
\node[] at (8.3125,-3.3) {$a_{2}a_{3}$};
\node[] at (11.083,-3.3) {$a_{3}A_{1}$};
\node[] at (13.413,-3.3) {$a_{3}A_{2}$};

\draw[ultra thick,dashed,->] (7,-0.6)--(6.5,-1.2);
\draw[ultra thick,dashed,->] (7,-0.6)--(7,-1.2);
\draw[ultra thick,dashed,->] (7,-0.6)--(7.5,-1.2);
\node[] at (7.6,-0.7) {$\mathbb I_{2}$};

\draw[ultra thick,dashed,->] (7,-2.1)--(7,-2.7);
\draw[ultra thick,dashed,->] (7,-2.1)--(7.5,-2.7);
\draw[ultra thick,dashed,->] (7,-2.1)--(6.5,-2.7);
\node[] at (7.6,-2.2) {$\mathbb I_{2}$};

\draw[ultra thick,dashed,->] (1.75,-2.1)--(2.25,-2.7);
\draw[ultra thick,dashed,->] (1.75,-2.1)--(1.25,-2.7);
\node[] at (2.35,-2.2) {$\mathbb I_{1}$};

\draw[ultra thick,dashed,->] (12.25,-2.1)--(12.75,-2.7);
\draw[ultra thick,dashed,->] (12.25,-2.1)--(11.75,-2.7);
\node[] at (12.85,-2.2) {$\mathbb I_{1}$};

\node[] at (7,-3.8) {$\vdots$};
\node[] at (7,-4.8) {$\vdots$};

\end{tikzpicture}
\caption{Generation of an $\infty$-variable Cantor set by applying the Iterated Function Systems $\mathbb I_{1}$ and $\mathbb I_{2}$ independently for every finite coding in the preceding level.}
\label{1vsInfiniteVar2}
\end{figure}

\begin{exam}
Figures~\ref{1vsInfiniteVar1} and~\ref{1vsInfiniteVar2} show the difference in construction of $1$-variable and $\infty$-variable RIFS. Both attractors are created by the two IFSs $\mathbb I_{1}=\{1/3x, 1/3x+2/3\}$ and $\mathbb I_{2}=\{1/4x,  1/4x+3/8, 1/4x+3/4\}$, with $\vec\pi=\{1/2,1/2\}$ but in the $1$-variable construction the IFS chosen is uniform on every level of the construction, whereas the $\infty$-variable attractor is not subject to this restriction.
The Hausdorff dimension of both of these attractors can be calculated to be almost surely $\dim_{H}F_{1\text{-var}}=0.721057$ and $\dim_{H}F_{\infty\text{-var}}=0.724952$ (both to 6 s.f.), see below.
\end{exam}

Perhaps contrary to first impression, the independence in $\infty$-variable attractors makes them easier to analyse and we shall give some results for the two settings below.
Assuming a random analogue of the OSC, the uniform open set condition (UOSC), we find that in the $\infty$-variable case the Hausdorff dimension is a.s.\ given by the unique $s$ satisfying
\[
\E_{i\in\Lambda}\left(\sum_{j\in\mathcal I_i}c_j^s\right)=1,
\]
(see Falconer~\cite{Falconer86}, Mauldin-Williams~\cite{Mauldin86} and Graf~\cite{Graf87}) whilst in the $1$-variable case it is a.s.\ the unique $s$ satisfying
\begin{equation}\label{1varEquation}
\E_{i\in\Lambda}\left(\log\sum_{j\in\mathcal I_i}c_j^s\right)=0,
\end{equation}
(see Hambly~\cite{Hambly92}). Further it has been observed that for the $\infty$-variable construction the Hausdorff and upper box dimension coincide almost surely, see Liu and Wu \cite{Liu02}. The latter result, and the equality of Hausdorff and upper box counting dimension for deterministic self-similar attractors of Falconer~\cite{Falconer89} are the main motivation for this manuscript, in which we prove that the Hausdorff and upper box dimension coincide, independent of overlap, almost surely. This generalises previously mentioned results and complements them, to give a more complete characterisation of this new model of random graph directed attractors, which naturally encompass the class of $1$-variable and $\infty$-variable (standard IFS) self-similar sets.
Furthermore we shall show that, in contrast to the Hausdorff, packing, and box counting dimension, the Assouad dimension is almost surely maximal in the sense that it is bounded below by an expression resembling the joint spectral radius of matrices rather than the Lyapunov exponents of random matrix multiplication.  This relates to earlier work of Fraser, Miao and Troscheit~\cite{Fraser14c} in which the Assouad dimension was found for different types of random IFS and percolation structures.
Here we find that some of the mentioned extensions in~\cite{Fraser14c} hold and the Assouad dimension is related to the joint spectral radius of a construction we shall introduce in Section~\ref{resultsSection}.

There is, of course, the natural question of an extension of deterministic graph directed attractors to a random version.
The usual model for this considers a fixed directed multi-graph, where for each edge we associate a family of maps with a probability measure and choose a map in a recursive fashion according to this probability measure. This  model was extensively studied in Olsen~\cite{Olsen94} and we refer to this book and the references contained therein.
Here we shall adopt a different natural model, that has so far not been considered in the literature but is nevertheless an object that arises in the study of sets that have orthogonal projections more complicated than for simple self-similar IFSs. 
Instead of one fixed graph, we consider a finite collection of graphs with an associated probability vector. We consider a $1$-variable random graph directed system (RGDS) and then a $\infty$-variable RGDS, where instead of the maps, the graphs and hence the relations between vertex sets changes in a random fashion. 
One example of sets whose projections fail to be self-similar RIFS but are random graph directed attractors in our sense, are the $V$-variable extensions of self-affine carpets in the sense of Fraser~\cite{Fraser12}. Failure here is caused by the non-trivial rotations and the projections cannot be described by the standard RIFS model but can be by the RGDS proposed here, see~\cite{Troscheit15a}.

It is worth noticing that many standard random models can be recovered by setting up the RGDS in the right way. 
Choosing graphs with a single vertex allows the RGDS set-up to be used to analyse $1$-variable and random recursive attractors. The class of $V$-variable attractors are specific $1$-variable RGDS in our sense, where one choses a vertex set with $V$ vertices and the $\Gamma_{i}$ with edges and probabilities appropriately.
Results about several other standard models can be deduced from our main theorems, see Corollary~\ref{VvarandOther}.
It is a quick calculation to show that $V$-variable constructions satisfy all conditions in Definition~\ref{graphDefs} and one can reduce the $V$-variable randomness to the simpler $1$-variable RGDA construction treated here. The model developed in this manuscript can be further generalised to $V$-variable RGDS and higher order random graph directed systems with the methods introduced here but we will not deal with the additional complexity of these constructions.
We also remark that in the $\infty$-variable case we are allowed to have paths that can become extinct, so choosing the graphs and maps appropriately our model specialises to fractal and Mandelbrot percolation.

We first give basic notation, define the model and give our main results for $1$-variable RGDS in Section~\ref{resultsSection}. Section~\ref{infVarSect} contains our $\infty$-variable results and proofs are contained in Section~\ref{proofsection}.

\section{Notation and preliminaries for $1$-variable RGDS}\label{resultsSection}
Let $\mathbf{\Gamma}=\{\Gamma_i\}_{i\in\Lambda}$ be a finite collection of graphs $\Gamma_i=\Gamma(i)=(V(i),E(i))$ indexed by $\Lambda=\{1,\hdots,n\}$, each with the same number of vertices. For simplicity we will assume that they share the same set of vertices $V(i)=V$.
The set $E(i)$ is the set of all directed edges and we write $\tensor*[_v]{E}{_w}(i)$ to denote the edges from $v\in V$ to $w\in V$. We write $\tensor*[]{E}{_w}(i)=\bigcup_{v\in V}\tensor*[_v]{E}{_w}(i)$ and $\tensor*[_v]{E}{}(i)=\bigcup_{w\in V}\tensor*[_v]{E}{_w}(i)$ for $i\in\Lambda$.
For all edges $e$ we write $\iota(e)$ and $\tau(e)$ to refer to initial and terminal vertex, respectively.
The set of all infinite strings with entries in $\Lambda$ we denote by $\Omega=\Lambda^{\N}$, whereas all finite strings of length $k$ are given by $\Omega^k=\Lambda^k$, and the set of all finite strings is $\Omega^*=\bigcup_{k\in\N}\Omega^k$. Elements in $\Omega^*$ and $\Omega$ are given by subscript notation, for $\omega\in\Omega$ we have $\omega=(\omega_1,\omega_2,\hdots)$ and for $\omega\in\Omega^*$ we have $\omega=(\omega_1,\omega_2,\hdots,\omega_l)$, where $l=\lvert\omega\rvert$ is the length of the string $\omega$. We define the $w\in\Omega^*$ cylinder in $\Omega$ to be the set of all infinite sequences starting with the finite word $w$. For $w\in\Omega^{*}$ we define the $w$-cylinder $[w]=\{\omega\in\Omega\mid \omega_{i}=w_{i}\text{ for }1\leq i\leq \lvert w\rvert\}$. 
We define a metric on $\Omega$ by $d(x,y)=2^{-\lvert x\wedge y\rvert}$, where $x\wedge y=z\in \Omega^{m}$ for $m=\max\{k\mid x_i=y_i\text{ for all }1\leq i\leq k\}$, and $z_{i}=x_{i}=y_{i}$ for all $1\leq i\leq m$, and $d(x,y)=0$ if no such $k$ exists. 
The metric induces a topology on $\Omega$ which is also generated by the cylinder sets, which are in fact clopen sets. We consider the shift map $\sigma(\omega_1,\omega_2,\hdots)=(\omega_2,\omega_3,\hdots)$ on $\Omega$. We can define a Bernoulli probability measure $\mu$ on $\Omega$ with probability vector ${\vec\pi}=\{\pi_1,\pi_2,\hdots,\pi_n\}$ first on all the cylinders $\omega\in\Omega^{*}$ by taking
\[
\mu([\omega])=\prod_{k=1}^{\lvert\omega\rvert}\pi_{\omega_{k}}.
\]
As the cylinders generate the topology of $\Omega$, the Carath\'eodory extension theorem implies that $\mu$ extends to a unique Borel measure on $\Omega$.

Given a collection of graphs $\mathbf{\Gamma}$ we are now interested in the attractor of two associated random processes. We first describe the $1$-variable case. For $v\in V$, we define the random attractor $K_{v}$ for $v\in V$ in terms of paths on the randomly chosen graphs. Let $\tensor*[_{v}]{E}{^{k}_{u}}(\omega)$ be the set of all paths of length $k$ consisting of edges starting at $v$ and ending at $u$ and traversing through the graph $\Gamma_{\omega_{q}}$ at step $q$, that is \label{sym:randompath}
\begin{multline}
\tensor*[_{v}]{E}{^{k}_{u}}(\omega)=\{\mathbf{e}=(e_{1},e_{2},\hdots,e_{k})  \mid   \iota(e_{1})=v, \tau(e_{k})=u, \iota(e_{l+1})=\tau(e_{l}) \\\text{ for }1\leq l \leq k-1 \text{ and }e_{i}\in E(\omega_{i})\}.\nonumber
\end{multline}
To each edge $e\in \{e\in E(i) \mid i\in\Lambda\}$ we associate a strictly contracting self-map $S_{e}:\R^{d}\to\R^{d}$ and choose a compact \emph{seed set} $\Delta\in\mathcal{K}(\R^{d})$ such that $\Delta=\overline{\interior \Delta}$ and $S_{e}(\Delta)\subset\Delta$ for all $e\in E(i)$ and $i\in\Lambda$.
In this notation we have
\[
K_{v}(\omega)=\bigcap_{l=1}^{\infty}\bigcup_{u\in V}\bigcup_{\mathbf{e}\in \tensor*[_{v}]{E}{^{l}_{u}}(\omega)}S_{\mathbf{e}}(\Delta),
\]
where $S_\mathbf{e}=S_{e_{1}}\circ S_{e_{2}}\circ\hdots\circ S_{e_{|\mathbf{e}|}}$. The set $K_{v}(\omega)$ is well-defined for every $\omega$ and $v$ and it is a simple application of Banach's fixed point theorem to show that $K_{v}(\omega)$ is compact and non-empty.
Even though this holds for all collections of contracting maps, we restrict our attention to similarities, i.e.\ maps such that $\lvert S_e (x)- S_e (y) \rvert = c_e \lvert x-y \rvert$ for some $0<c_e < 1$ and all $x,y\in \R^d$.

We denote the expectation of a random variable $X:Z\to\R$, where $Z$ is the space of all possible outcomes (realisations) $z\in Z$, by $\E_{z}X(z)=\int_{Z}X(z)\,d\nu(z)$, with $\nu$ an appropriate probability measure on $Z$. The probability that an event $\mathfrak F$ occurs shall be denoted by $\Prob(\mathfrak F)$ and we write $\E^{\geo}$ for the \emph{geometric expectation} 
\[
\E^{\geo}_{z}X(z)=\exp\int_{Z}\log X(z)\,d\nu(z).
\]
We shall leave out the subscript from the expectation notation if it is clear from context which space of outcomes we are considering.

We will refer to the Hausdorff, packing, Assouad, upper and lower box counting dimension by $\dim_{H}$, $\dim_{\text{P}}$, $\dim_{A}$, $\overline\dim_{B}$, $\underline\dim_{B}$, respectively. If the box counting dimension exists we shall refer to it as simply $\dim_{B}$.

\begin{defn}
Let $F\subseteq\R^{d}$ and $s\in\R^{+}_{0}$, we define the $s$-dimensional Hausdorff $\delta$-premeasure of $F$ by
\[
\Haus^{s}_{\delta}(F)=\inf\left\{ \sum_{k=1}^{\infty} \lvert U_{k} \rvert^{s} \mid\{U_{i}\}\text{ is a $\delta$-cover of $F$} \right\},
\]
where the infimum is taken over all countable $\delta$-covers and $\lvert U\rvert$ refers to the diameter of $U$. The $s$-dimensional Hausdorff measure of $F$ is then 
\[
\Haus^{s}(F)=\lim_{\delta\to0}\Haus^{s}_{\delta}(F).
\] 
The \emph{Hausdorff dimension} is defined to be 
\[
\dim_{H}F=\inf\{s \mid \Haus^{s}(F)=0\}.
\]
\end{defn}

\begin{defn}
Let $X\subseteq \R^{d}$ be a totally bounded set and $N_{\epsilon}(X)$ be the smallest number of sets of diameter $\epsilon$ or less needed to cover $X$. The upper and lower box counting dimensions of $X$ are, respectively, given by
\[
\overline\dim_{B}X=\limsup_{\epsilon\to0}\frac{\log N_{\epsilon}(X)}{-\log\epsilon}
\hspace{0.3cm}\text{ and }\hspace{0.3cm}
\underline\dim_{B}X=\liminf_{\epsilon\to0}\frac{\log N_{\epsilon}(X)}{-\log\epsilon}.
\]
If the limit exists we refer to the box counting dimension as
\[
\dim_{B}X=\lim_{\epsilon\to0}\frac{\log N_{\epsilon}(X)}{-\log\epsilon}.
\]

\end{defn}

\begin{defn}
Let $X\subseteq \R^{d}$, we define the Assouad dimension of $X$ to be
\begin{multline*}
\dim_A X \ = \  \inf \Bigg\{ \alpha \  : \ \text{     there exists a constant $C>0$ such that,} \\
\text{ for all $0<r<R<\infty$, we have $\ \sup_{x \in F} \, N_r\big( B(x,R) \cap F \big) \ \leq \ C \bigg(\frac{R}{r}\bigg)^\alpha$ } \Bigg\}.
\end{multline*}
\end{defn}

As we will not directly deal with packing dimension we omit the definition. A detailed introductory treatment to the classical notions of fractal dimension (Hausdorff, packing and box counting) is Falconer~\cite{FractalGeo}. For a summary of properties of the Assouad dimension see Fraser~\cite{Fraser14a}. In particular
\[
\dim_{H}F\leq\underline\dim_{B}F\leq\overline\dim_{B}F\leq\dim_{A}F.
\]

In many places the $1$-variable result depends on a structure that is an infinite matrix over finite matrices with (semi-)ring element entries. Let $\mathcal M_{n\times n}(\R)$ be the space of all $n\times n$ matrices with real entries and $\mathcal M_{n\times n}(\R^{+}_{0})$ the set of all $n\times n$ matrices with non-negative entries.
We shall also consider the set of square matrices with entries that are finite non-negative matrices
\[
\mathfrak M_{k,n}=\mathcal M_{k\times k}(\mathcal M_{n\times n}(\R^{+}_{0})),
\]
 and the (vector) space of countably infinite, upper triangular matrices with entries that are finite real-valued matrices 
\[
\mathfrak M^{*}_{\N,n}=\mathcal M_{\N\times \N}(\mathcal M_{n\times n}(\R)),
\]
such that for every $M\in\mathfrak M^{*}_{\N,n}$ the number of row entries that are not the zero matrix is uniformly bounded and
\begin{equation}\label{eqn:maxithing}
\sup_{j\in\N}\sum_{i=0}^{\infty}\lVert M_{i,j}\rVert_{row}<\infty,
\end{equation}
where $\lVert.\rVert_{row}$ is the matrix norm, see below.
It can be checked that $\mathfrak M^{*}_{\N,n}$ is a vector space and we  consider the subset consisting of non-negative entries
\[
\mathfrak M_{\N,n}=\mathcal M_{\N\times \N}(\mathcal M_{n\times n}(\R^{+}_{0}))\subset \mathfrak M^{*}_{\N,n}.
\]
We note that the only infinite matrices we are considering are upper triangular.
Further, while the sets $\mathfrak M_{k,n}$ and $\mathfrak M_{\N,n}$ are not vector spaces per se, they are subsets of vector spaces that are closed under multiplication and addition. We define the following norms and seminorms.
\begin{defn}
Let $M\in\mathcal M_{n\times n}(\R)$, we define
\[
\lVert M \rVert_{\text{row}}=\max_{i}\sum_{j}\lvert M_{i,j}\rvert
\]
\[
\lVert M \rVert_{1}=\sum_{i}\sum_{j}\lvert M_{i,j}\rvert
\]
which can easily seen to be (equivalent) norms. For $M^{*}\in\mathfrak{M}^{*}_{\N,n}$, the space of infinite matrices consisting of matrix entries with  real entries, such that only finitely many matrices in each row are not $\mathbf{0}$ (the $n\times n$ zero matrix) and the norm of each row sum is uniformly bounded, we define the norm
\[
\vertiii{M^{*}}_{\sup}=\sup_{i^{*}\in\N}\sum_{j^{*}=1}^{\infty}\lVert(M^{*})_{i^{*},j^{*}}\rVert_{\text{row}}.
\]
Furthermore we define two seminorms. The first $\vertiii{\mathds{1}\,.\,}$ is given by (\ref{infiniteSeminorm}) and defined on the same space $\mathfrak{M}_{\N,n}^*$ of infinite matrices with real-valued matrix entries such that the number of non-zero matrix entries is uniformly bounded above and (\ref{eqn:maxithing}) is satisfied. The second seminorm $\vertiii{\mathds{1}_{l}\,.\,}_{(1,1)}$, given by (\ref{square1norm}), is defined on the space of $l$ by $l$ matrices with $n$ by $n$ real matrix entries.
We slightly abuse notation here and concisely write $\lVert\mathbf{v}\rVert_{s}$, where $\mathbf{v}$ is a vector with matrix entries, to mean the matrix sum of all, possibly infinite, vector entries.
 Here $\mathds{1}=\{\mathbf{1},\mathbf{0},\mathbf{0},\dots\}$ is an infinite vector and $\mathds{1}_{l}$ is the vector of dimension $l$ satisfying $\mathds{1}_{l}=\{\mathbf{1},\mathbf{0},\mathbf{0},\dots,\mathbf{0}\}$, where $\mathbf{1}$ is the $n\times n$ identity matrix.
\begin{equation}\label{infiniteSeminorm}
\vertiii{\mathds{1}M}=\lVert\lVert\mathds{1}M\rVert_{s}\rVert_{\text{row}}=\left\lVert \sum_{k=1}^{\infty} (\mathds{1}M)_{k} \right\rVert_{\text{row}}
\end{equation}
\begin{equation}\label{square1norm}
\vertiii{\mathds{1}_{l}M}_{(1,1)}=\lVert\lVert\mathds{1}_{l}M\rVert_{s}\rVert_{1}=\sum_{i,j\in\{1,\dots,n\}} \sum_{k=1}^{l} ((\mathds{1}_{l}M)_{k})_{i,j}
\end{equation}
\end{defn}

Before we introduce further necessary notation we refer the reader to two important corollaries of our more general results. 
First, in Corollary~\ref{separatedCorollary} we state the almost sure Hausdorff dimension of our $1$-variable random graph directed systems, assuming the uniform strong separation condition. The quantity $p^{t}_{1}(\omega,1)$ referred to in~(\ref{USSCCorFormula}) is simply the Hutchinson-Moran matrix for the graph-directed iterated function system associated with $\Gamma(\omega_{1})$.
Furthermore Corollary~\ref{VvarandOther} states that for self-similar $1$-variable sets, and even $V$-variable sets in the sense of Barnsley et al.~\cite{Barnsley08}, we must have $\dim_{H}F_{\omega}=\dim_{B}F_{\omega}$ for almost every $\omega\in\Omega$.

\subsection{Arrangements of words}\label{wordConstructionSection}
To describe the cylinders and points in the attractor of Iterated Function Systems and Graph Directed Systems, one uses a natural coding. 
In this section we shall give a more abstract way of manipulating words that will become useful in describing the construction in random systems.
We introduce two binary operations $\oplus$ and $\odot$ that take over the r\^oles of set union and concatenation, respectively, to manipulate strings in a meaningful way.
\begin{defn}
Let $\mathcal{G}^{E}$ be a finite alphabet, which in this article is the set of letters identifying the edges of the graphs $\Gamma_{i}$, i.e.\ $\mathcal{G}^{E}=\{e \mid e\in E(i)\text{ and }i\in\Lambda\}$. We define the \emph{prime arrangements} $\mathcal{G}$ to be the set of symbols $\mathcal{G}=\{\varnothing,\epsilon_{0}\}\cup\mathcal{G}^{E}$. Clearly both $\mathcal{G}$ and $\mathcal{G}^{E}$ are finite and non-empty.

Define $\beth^{\odot}$ to be the free monoid with generators $\mathcal{G}^{E}$ and identity $\epsilon_{0}$, and define $\beth^{\oplus}$ to be the free commutative monoid with generators $\beth^{\odot}$ and identity $\varnothing$. We define $\odot$ to be left and right multiplicative over $\oplus$, and $\varnothing$ to annihilate with respect to $\odot$. That is, given an element $e$ of $\beth^{\odot}$, we get $e\odot\varnothing=\varnothing\odot e=\varnothing$. 
We define $\beth^{*}$ be the set of all finite combinations of elements of $\mathcal{G}$ and operations $\oplus$ and $\odot$. Using distributivity $\beth=(\beth^{*},\oplus,\odot)$ is the non-commutative free semi-ring with `addition' $\oplus$ and `multiplication' $\odot$ and generator $\mathcal{G}^{E}$ and we will call $\beth$ the \emph{semiring of arrangements of words} and refer to elements of $\beth^{*}$ as \emph{(finite) arrangements of words}. 
\end{defn}
We adopt the convention to `multiply out' arrangements of words and write them as elements of $\beth^{\odot}$. Furthermore we omit brackets, where appropriate, replace $\odot$ by concatenation to simplify notation, and for arrangements of words $\phi$ write $\varphi\in\phi$ to refer to the maximal subarrangements $\varphi$ that do not contain $\oplus$ and are thus elements of $\varphi\in\beth^{\odot}$.

\begin{exam}
Let $\mathcal{G}^{E}=\{0,1\}$. The set of prime arrangements is then $\{\varnothing, \epsilon_{0}, 0,1\}$ and the elements of the semiring $\beth^{*}$ are all possible concatenations $\odot$ and unions $\oplus$, e.g.
\[
1\odot0\oplus1=10\oplus1, \;\;(110\oplus101\oplus\epsilon_{0})\odot1=1101\oplus1011\oplus1,\;\; \varnothing\odot(10\oplus101)=\varnothing,\dots
\]
\end{exam}
The usefulness of the description above is that $\beth^{*}$ is ring isomorphic to the set of all cylinders with set union and concatenation as the binary operations and we can use $\odot$ and $\oplus$ to describe collections of cylinders. For example the set containing all cylinders of length $k$ can be identified with the arrangement of words $(0\oplus 1)^{k}$.

We can now use the algebraic structure above to give descriptions of 1-variable RIFS. 

\begin{exam}
Consider the simple setting of just two Iterated Functions Systems $\mathbb L=\{\mathbb I_{1}, \mathbb I_{2}\}$ that are picked at random according to probability vector $\vec\pi=\{\pi_{1},\pi_{2}\}$, $\pi_{i}>0$. Let $\phi_{i}=a^{i}_{1}\oplus\dots\oplus a^{i}_{n}$, where $a_{j}^{i}$ are the letters in the alphabet associated with IFS $\mathbb{I}_{i}$. The arrangement of words describing the cylinders of length $k$ with realisation $\omega$ is then simply 
\[
\phi_{\omega_{1}}\odot\phi_{\omega_{2}}\odot\dots\odot\phi_{\omega_{k}}.
\]
\end{exam}

Before we can apply this construction to our RGDS  we need to extend this concept to the natural analogue of matrix multiplication $\otimes$ and addition, which we shall also refer to as $\oplus$.

\begin{defn}
Let $\mathbf{M}$ and $\mathbf{N}$ be square $n\times n$ matrices and $\mathbf{v}=\{v_{1},\dots,v_{n}\}$ be a $n$-vector with entries being arrangements of words. We define matrix multiplication in the natural way,
\[
(\mathbf{M}\otimes\mathbf{N})_{i,j}=\bigoplus_{k=1}^{n}(\mathbf{M}_{i,k}\odot\mathbf{N}_{k,j}),\hspace{0.5cm}
(\mathbf{M}\oplus\mathbf{N})_{i,j}=\mathbf{M}_{i,j}\oplus\mathbf{N}_{i,j},
\]
\[
(\mathbf{v}\otimes\mathbf{M})_{i}=\bigoplus_{k=1}^{n}({v}_{k}\odot\mathbf{M}_{k,i}).
\]
\end{defn}\label{wordMultiplicationdef}
We extend this to multiplication of countable (finite or infinite) square matrices with matrix entries.
\begin{defn}
Let $\mathbf{M^{*}}$ and $\mathbf{N^{*}}$ be elements of $\mathcal{M}_{k,k}(\mathcal{M}_{n,n}((\beth^{*}))$ and $\mathbf{v}^{*}\in (\mathcal{M}_{n,n}((\beth^{*}))^{k}$, where $k\in\N\cup\{\N\}$. We define multiplication and addition by 
\[
(\mathbf{M^{*}}\otimes\mathbf{N^{*}})_{i,j}=\bigoplus_{l=1}^{k}(\mathbf{M^{*}}_{i,l}\otimes\mathbf{N^{*}}_{l,j}),\hspace{0.5cm}
(\mathbf{M^{*}}\oplus\mathbf{N^{*}})_{i,j}=\mathbf{M^{*}}_{i,j}\oplus\mathbf{N^{*}}_{i,j},
\]
and
\[
(\mathbf{v}^{*}\otimes\mathbf{M}^{*})_{i}=\bigoplus_{l=1}^{k}(\mathbf{v}^{*}_{l}\otimes \mathbf{M}^{*}_{l,i}).
\]
\end{defn}

\subsection{Stopping graphs}\label{stoppingGraphs}
We continue this section by introducing the notion of the $\epsilon$-stopping graph. Before we can do so we need some conditions on our graphs $\mathbf{\Gamma}$.
\begin{defn}\label{graphDefs}
Let $\mathbf{\Gamma}=\{\Gamma_{i}\}_{i\in\Lambda}$ be a finite collection of graphs, sharing the same vertex set $V$. 
\begin{enumerate}[label=\arabic{section}.\arabic{theo}.\alph*]
\item\label{nontrivialgraph} We say that the collection $\mathbf{\Gamma}$ is a \emph{non-trivial collection of graphs} if for every $i\in\Lambda$ and $v\in V$ we have $\tensor*[_{v}]{E(i)}{}\neq\varnothing$.  Furthermore we require that there exist $i,j\in\Lambda$ and $e_{1}\in\Gamma(i)$ and $e_{2}\in \Gamma(j)$ such that $S_{e_{1}}\neq S_{e_{2}}$.

\item\label{strongConnected} If for every $v,w\in V$ there exists $\omega^{v,w}\in\Omega^{*}$ such that $\tensor[_{v}]{E}{_{w}}(\omega^{v,w})\neq\varnothing$ and $\mu([\omega^{v,w}])>0$, we call $\mathbf{\Gamma}$ \emph{stochastically strongly connected}.

\item\label{contractingSSRGDS} We call the \emph{Random Graph Directed System (RGDS)} associated with $\mathbf{\Gamma}$ a \emph{contracting self-similar RGDS} if for every $e\in E(i)$, $S_{e}$ is a contracting similitude.
\end{enumerate}
\end{defn}
Condition \ref{strongConnected} implies that at each stage of the construction there is a positive probability that one can travel from every vertex to every other in a finite number of steps.
As every map for every edge in $\mathbf{\Gamma}$ is a strict contraction the maximal similarity coefficient $c_{\max}=\max\{c_\textbf{e}\mid \textbf{e}\in E(i)\text{ and }i\in\Lambda\}$ satisfies $c_{\max}<1$. This gives us that for every $\epsilon>0$ there exists a least $k_{\max}(\epsilon)\in\N$ such that $c_{\max}^{k_{\max}}<\epsilon$ and hence every path $\mathbf{e}\in\tensor*[]{E}{^{k_{\max}}}(\omega)$ has an associated contraction $c_\mathbf{e}<\epsilon$.
Therefore all paths of length comparable with $\epsilon$ only depend, at most, on the first $k_{\max}(\epsilon)$ letters of the random word $\omega\in\Omega$ and thus the set of $\epsilon$-stopping graphs below is well defined.
\begin{defn}\label{stoppingGraphDefn}
Let $\mathbf{\Gamma}$ be a non-trivial, finite collection of graphs sharing vertex set $V$, satisfying Condition~\ref{contractingSSRGDS}. Let $E^{*}(\omega,\epsilon)$ be the set of paths $\mathbf{e}$, corresponding to the realisation $\omega$, such that $S_\mathbf{e}$ is a contraction with similarity coefficient comparable to $\epsilon$:
\begin{multline}
E^{*}(\omega,\epsilon)=\left\{\mathbf{e}\in\bigcup_{k=1}^{k_{\max}(\epsilon)}\tensor*[]{E}{^{k}}(\omega) \mid c_\mathbf{e}\leq\epsilon \text{ for }\mathbf{e}=(e_{1},\hdots,e_{\lvert\mathbf{e}\rvert})\right.\\
\left.\vphantom{\bigcup_{k=1}^{k_{\max}(\epsilon)}\tensor*[]{E}{^{k}}}
\text{ but }c_{\mathbf{e}^{\ddagger}}>\epsilon\text{ for }\mathbf{e}^{\ddagger}=(e_{1},\hdots,e_{\lvert\mathbf{e}\rvert-1})\right\}.\nonumber
\end{multline}
Now consider all possible subsets of these sets of edges $\mathcal{E}(\omega,\epsilon)$, such that the images of $\Delta$ are pairwise disjoint in each of the subsets
\[
\mathcal{E}(\omega,\epsilon)=\{U \subseteq E^{*}(\omega,\epsilon) \mid \text{for all } \mathbf{e},\mathbf{f}\in U\text{ we have } S_\mathbf{e}(\Delta)\cap S_\mathbf{f}(\Delta)=\varnothing\}.
\]
As $\mathcal{E}(\omega,\epsilon)$, and every $U_{i}\in\mathcal{E}(\omega,\epsilon)$, has finite cardinality we can order $\{U_{i}\}$ in descending order, i.e.\ $\lvert U_{m}\rvert\geq \lvert U_{m+1}\rvert$. Finally we define  $E(\omega,\epsilon)$ to be the first, and thus maximal, element $E(\omega,\epsilon)=U_{0}$.\\
The \emph{$\epsilon$-stopping graph} is then defined to be
\[
\mathbf{\Gamma}^{\epsilon}=\{\Gamma^{\epsilon}(\omega) \mid z_{i}\in\Lambda^{k_{\max}(\epsilon)}\text{ and } \omega\in[z_{i}]\}\text{, with }
\Gamma^{\epsilon}(\omega)=(V,E(\omega,\epsilon)).
\]
In fact it does not matter which $\omega\in[z_{i}]$ is chosen as $\Gamma^{\epsilon}(\omega)$ only depends on, at most, the first $k_{\max}(\epsilon)$ letters.
\end{defn}
 By the arguments above it can easily be seen that the collection $\mathbf{\Gamma}^{\epsilon}$ is finite for every $\epsilon>0$ and every edge of $\mathbf{\Gamma}^{\epsilon}$ is a finite path in $\mathbf{\Gamma}$ for the same $\omega$. However there may be some paths in $\mathbf{\Gamma}$ that are not edges of $\mathbf{\Gamma}^{\epsilon}$ for any $\epsilon$, but for $\epsilon$ small enough, eventually that path will be a prefix of an edge coding.
\begin{lma}\label{lma:stoppingrepl}
For every realisation $\omega$ the set of edges in $\mathbf{\Gamma}^{\epsilon}$ forms a stopping set. That is, for every path in $\mathbf{\Gamma}$ there exists an $\epsilon>0$ such that the path exists in $\mathbf{\Gamma}^{\epsilon}$ (although it may only be a prefix of a path) such that for every $\epsilon$ the collection $\mathbf{\Gamma}^{\epsilon}$ is finite as well as the edge set of every $\Gamma^{\epsilon}_{i}\in\mathbf{\Gamma}^{\epsilon}$.
\end{lma}
We will be considering $\epsilon$-stopping graphs derived from the original graph and show that if ${\Gamma}$ has `nice' properties (it satisfies most assumptions in Definition~\ref{graphDefs}), then ${\Gamma}^{\epsilon}$ also has these properties.
\begin{lma}
Let $\mathbf{\Gamma}$ be a non-trivial collection of graphs that is stochastically strongly connected. Then there exists $\epsilon'>0$ such that $\mathbf{\Gamma}^{\epsilon}$ is a non-trivial collection of stochastically strongly connected graphs for all $0<\epsilon\leq\epsilon'$ and almost every $\omega\in\Omega$.
\end{lma}
\begin{proof}
Assuming $\mathbf{\Gamma}$ is non-trivial implies that for every
$v\in V$ and $i\in\Lambda$ there exists at least one edge in $\tensor*[_{v}]{E}{}(i)$. However as the set $E(\omega,\epsilon)$ is chosen by non-overlapping images, for a path to be deleted there must be a second path, leaving at least one path. Hence $\lvert\tensor*[_{v}]{E}{}(\omega,\epsilon)\rvert\geq 1$ for all $v$ and $\omega$, i.e.\ $\mathbf{\Gamma}^{\epsilon}$ is non-trivial.
 
To show that $\mathbf{\Gamma}^{\epsilon}$ is stochastically strongly connected we note that the only possibility for a path that existed in $\mathbf{\Gamma}$ but not in  $\mathbf{\Gamma}^{\epsilon}$ is that it had been deleted due to overlapping images. However if $\epsilon$ is chosen small enough then there will be a different path that is being kept, unless all maps $S_{e}$ are identical. We however exclude this trivial case (Condition~\ref{nontrivialgraph}) as the attractor of such a system would be a singleton.
\end{proof}
We can partition the paths in $E(\omega,\epsilon)$ by initial and terminal vertex and path length and write $\tensor*[_{v}]{E}{_{w}^{k}}(\omega,\epsilon)$ to refer to paths $\mathbf{e}$ of length $k$ with $1\leq k\leq k_{\max}(\epsilon)$, $\iota(\mathbf{e})=v$ and $\tau(\mathbf{e})=w$. The set  $E(\omega,\epsilon)$ then consists of collections of paths whose images are disjoint under $S_\mathbf{e}$.

\subsection{Infinite random matrices}\label{infiniteMatrixSection}
Let $\omega\in\Omega$ be a word chosen randomly according to the Bernoulli measure $\mu$ associated with the probability vector ${\vec\pi}$, where $\pi_{i}>0$ for all $i\in\Lambda$. For all $i\in\{1,2,\hdots,l\}$ let ${t_{i}(\omega)}\in\mathcal M_{n\times n}(\R^{+}_{0})$. Letting $\mathbf{t}(\omega)=\{t_{1}(\omega),t_{2}(\omega),\hdots,t_{l}(\omega)\}$ we have a random vector with matrix valued entries.
Now define $\bf{T}(\omega)\in\mathfrak M_{\N,n}$ by
\[
\bf{T}(\omega)=
\begin{pmatrix}
t_{1}(\omega)	&	\mathbf{0}				&		\mathbf{0}				&	\mathbf{0}				&	\hdots	\\
t_{2}(\omega)	&	t_{1}(\sigma\omega)		&		\mathbf{0}\vphantom{\vdots}	&	\mathbf{0}				&	\hdots	\\
\vdots		&	t_{2}(\sigma\omega)		&		t_{1}(\sigma^{2}\omega)	&	\mathbf{0}				&	\hdots	\\
t_{l}(\omega)	&	\vdots				&		t_{2}(\sigma^{2}\omega)	&	t_{1}(\sigma^{3}\omega)	&	\ddots	\\
\mathbf{0}		&	t_{l}(\sigma\omega)		&		\vdots				&	t_{2}(\sigma^{3}\omega)	&			\\
\mathbf{0}		&	\mathbf{0}				&		t_{l}(\sigma^{2}\omega)	&	\vdots				&	\ddots	\\
\vdots		&	\mathbf{0}				&		\mathbf{0}				&	t_{l}(\sigma^{3}\omega)	&			\\
\vdots		&	\vdots				&		\vdots				&	\vdots				&	\ddots	\\
\end{pmatrix}^{\top}.
\]
The transpose in the definition above is solely to represent $\mathbf{T}$ in a more readable fashion. We also, as indicated in 
Definition~\ref{wordMultiplicationdef}, construct matrices consisting of collections of words. 
For the $1$-variable construction we need two different constructions: a finite and an infinite version corresponding to the $\epsilon$-stopping graph defined in Definition~\ref{stoppingGraphDefn}. 
We only give the infinite construction here as it is needed to state our main results. Since the finite version is only used in the proof of Theorem~\ref{lowerHausdorff} we postpone its definition until then.
Let $\mathbf{\Gamma}^{\epsilon}$ be given and consider the partition of edges of $\Gamma^{\epsilon}(i)$ into the sets $\tensor*[_{v}]{E}{^{k}_{w}}(\omega,\epsilon)$. We assign unique letters to each of the paths of $\mathbf{\Gamma}$ that are now the edges of the graphs $\mathbf{\Gamma}^{\epsilon}$.  For $V=\{1,\dots,n\}$, let  $\eta$ be a  $n\times n$ matrix over arrangements of words that are collections of these letters representing the edges. We let, for $1\leq q \leq k_{\max}(\epsilon)$, 
\[
\eta_{q}(\omega,\epsilon)=
\begin{pmatrix}
\bigoplus_{e\in (\tensor*[_{1}]{E}{^{q}_{1}}(\omega,\epsilon))}e	\vphantom{\vdots}&	\bigoplus_{e\in (\tensor*[_{1}]{E}{^{q}_{2}}(\omega,\epsilon))}e	&	\hdots	&	\bigoplus_{e\in (\tensor*[_{1}]{E}{^{q}_{n}}(\omega,\epsilon))}e\\
\bigoplus_{e\in (\tensor*[_{2}]{E}{^{q}_{1}}(\omega,\epsilon))}e	&	\ddots	&	&	\vdots\\
\vdots	&	&	\ddots	&	\vdots\\
\bigoplus_{e\in (\tensor*[_{n}]{E}{^{q}_{1}}(\omega,\epsilon))}e	&	\bigoplus_{e\in (\tensor*[_{n}]{E}{^{q}_{2}}(\omega,\epsilon))}e	&	\hdots	&	\bigoplus_{e\in (\tensor*[_{n}]{E}{^{q}_{n}}(\omega,\epsilon))}e
\end{pmatrix}
\]
We also need to refer to the two elements corresponding to the identity and zero matrix in this setting. Let $\mathbf{0}_{\varnothing}$ and $\mathbf{1}_{\epsilon_{0}}$ be $n\times n$ matrices such that
\[
(\mathbf{0}_{\varnothing})_{i,j}=\varnothing\;\;\text{ and }\;\; (\mathbf{1}_{\epsilon_{0}})_{i,j}=
\begin{cases}
\epsilon_{0}&\text{ if }i=j,\\
\varnothing & \text{ otherwise}.
\end{cases}
\]
Furthermore let $\widehat\eta^{i}(\omega,\epsilon)=\{\mathbf{0}_{\varnothing}, \dots,\mathbf{0}_{\varnothing}, \eta_{1}(\omega,\epsilon),\dots \eta_{k_{\max}(\epsilon)}(\omega,\epsilon),\mathbf{0}_{\varnothing},\mathbf{0}_{\varnothing},\dots\}$: the edges in the partitions arranged by length of original paths and prefixed by $i-1$ occurrences of $\mathbf{0}_{\varnothing}$. 
The matrix $\mathbf{H}^{\epsilon}(\omega)$ has row entries given by the vectors $\widehat\eta^{i}(\omega,\epsilon)$, in particular the $k$-th row of $\mathbf{H}^{\omega}(\epsilon)$ is $\widehat\eta^{k}(\sigma^{k-1}\omega,\epsilon)$ for $k\geq 0$:
\[
(\mathbf{H}^{\epsilon}(\omega))_{i,j}=(\widehat\eta^{i}(\sigma^{i-1}\omega,\epsilon))_{j}.
\]

We need the structure as described above to construct the words with the stopping graph. The original attractors to $\mathbf{\Gamma}$ do not require this structure as words are constructed by multiplying 
\[
\eta_{1}(\omega_{1},1)\eta_{1}(\omega_{2},1)\dots\eta_{k-1}(\omega_{k-1},1)
\] 
and then taking the union over each row. However, when taking the $\epsilon$-stopping graph for non-trivial $\epsilon$ we have the added complication that edges in $\mathbf{\Gamma}^{\epsilon}$ arise from paths of potentially different lengths in $\mathbf{\Gamma}$. This needs to be considered when applying another edge as it does not only need to start with the correct vertex (the terminal vertex of the previous edge), but also on the length of the equivalent path in $\mathbf{\Gamma}$ such that the edges of the correct graph are applied, namely for an edge of length $k$ at iteration step $i$, the graph with realisation $\sigma^{k+i+1}\omega$ has to be used.
Writing this in terms of matrix notation makes sense as the row a word sits in relates to how long the path was that created it, so that when multiplying with the next random matrix, the correct graph $\Gamma_{i}$ is applied.
It can help to visualise this construction of words in a layered iterative fashion, see Figure~\ref{layeredWords}. Given $\omega\in\Omega$ one starts with the identity empty word matrix $\mathbf{1}_{\epsilon_{0}}$ and applies the first set of matrices $\{\eta_{i}(\omega)\}$ to it to get a collection of $k_{\max}(\epsilon)$ entries (the second row in the figure). The next row is obtained by applying $\{\eta_{i}(\sigma\omega)\}$ to the collection of words in the first entry, $\{\eta_{i}(\sigma^{2}\omega)\}$ to the second, etc., taking $\oplus$ unions when necessary.
The $k$-th entry of the $i$-th row corresponds to the collection of words $(\mathds{1}_{\epsilon}\mathbf{H}^{\epsilon}(\omega)\dots\mathbf{H}^{\epsilon}(\sigma^{i-1}\omega))_{k}$, where the vector $\mathds{1}_{\epsilon}\mathbf{H}^{\epsilon}(\omega)\dots\mathbf{H}^{\epsilon}(\sigma^{i-1}\omega)$ is the $i$-th row.

\begin{figure}
\label{layeredWords}
\begin{tikzpicture}[x=0.95cm, y=0.95cm]
\node[state, scale=0.7] at (0,0) (00) {$(1,1)$};

\node[state, scale=0.7] at (0,-1.5) (01){$(2,2)$};
\node[state, scale=0.7] at (1,-1.5) (11){$(3,2)$}; 
\node[state, scale=0.7] at (2,-1.5) (21){$(4,2)$};

\node[state, scale=0.7] at (0,-3) (02){$(3,3)$};
\node[state, scale=0.7] at (1,-3) (12){$(4,3)$}; 
\node[state, scale=0.7] at (2,-3) (22){$(5,3)$};
\node[state, scale=0.7] at (3,-3) (32){$(6,3)$};
\node[state, scale=0.7] at (4,-3) (42){$(7,3)$};

\node[state, scale=0.7] at (0,-4.5) (03){$(4,4)$};
\node[state, scale=0.7] at (1,-4.5) (13){$(5,4)$}; 
\node[state, scale=0.7] at (2,-4.5) (23){$(6,4)$};
\node[state, scale=0.7] at (3,-4.5) (33){$(7,4)$};
\node[state, scale=0.7] at (4,-4.5) (43){$(8,5)$};
\node[state, scale=0.7] at (5,-4.5) (53){$(9,5)$};
\node[state, scale=0.7] at (6,-4.5) (63){$(10,5)$};

\node[state, scale=0.7] at (0,-6) (04){$(5,6)$};
\node[state, scale=0.7] at (1,-6) (14){$(6,6)$}; 
\node[state, scale=0.7] at (2,-6) (24){$(7,6)$};
\node[state, scale=0.7] at (3,-6) (34){$(8,6)$};
\node[state, scale=0.7] at (4,-6) (44){$(9,6)$};
\node[state, scale=0.7] at (5,-6) (54){$(10,6)$};
\node[state, scale=0.7] at (6,-6) (64){$(11,6)$};
\node[state, scale=0.7] at (7,-6) (74){$(12,6)$};
\node[state, scale=0.7] at (8,-6) (84){$(13,6)$};

\path[->,thick]
(00) edge node {}(01)
(01) edge node {}(02)
(11) edge node {}(12)
(21) edge node {}(22)
(02) edge node {}(03)
(12) edge node {}(13)
(22) edge node {}(23)
(32) edge node {}(33)
(42) edge node {}(43)
(03) edge node {}(04)
(13) edge node {}(14)
(23) edge node {}(24)
(33) edge node {}(34)
(43) edge node {}(44)
(53) edge node {}(54)
(63) edge node {}(64)
;

\path[->,thick, dashed]
(00) edge node {}(11)
(01) edge node {}(12)
(11) edge node {}(22)
(21) edge node {}(32)
(02) edge node {}(13)
(12) edge node {}(23)
(22) edge node {}(33)
(32) edge node {}(43)
(42) edge node {}(53)
(03) edge node {}(14)
(13) edge node {}(24)
(23) edge node {}(34)
(33) edge node {}(44)
(43) edge node {}(54)
(53) edge node {}(64)
(63) edge node {}(74)
;

\path[->,thick, dotted]
(00) edge node {}(21)
(01) edge node {}(22)
(11) edge node {}(32)
(21) edge node {}(42)
(02) edge node {}(23)
(12) edge node {}(33)
(22) edge node {}(43)
(32) edge node {}(53)
(42) edge node {}(63)
(03) edge node {}(24)
(13) edge node {}(34)
(23) edge node {}(44)
(33) edge node {}(54)
(43) edge node {}(64)
(53) edge node {}(74)
(63) edge node {}(84)
;

\node[state] at (7,-0.2) (A) {$(i_{0},k)$};
\node[state] at (7,-3.2) (B) {$(i_{0},k')$};
\node[state] at (8.5,-3.2) (C) {$(i_{1},k')$};
\node[state] at (12,-3.2) (D) {$(i_{2},k')$};

\path[->,thick]
(A) edge node [left] {$\eta_{1}(\sigma^{i_{0}-1}\omega,\epsilon)$} (B);

\path[->,thick, dashed]
(A) edge node [right, near end] {$\eta_{2}(\sigma^{i_{0}-1}\omega,\epsilon)$} (C);

\path[->,thick, dotted]
(A) edge node [right] {$\eta_{3}(\sigma^{i_{0}-1}\omega,\epsilon)$} (D);

\end{tikzpicture}
\caption{Layered construction of words for $k_{\max}(\epsilon)=3$.}
\end{figure}
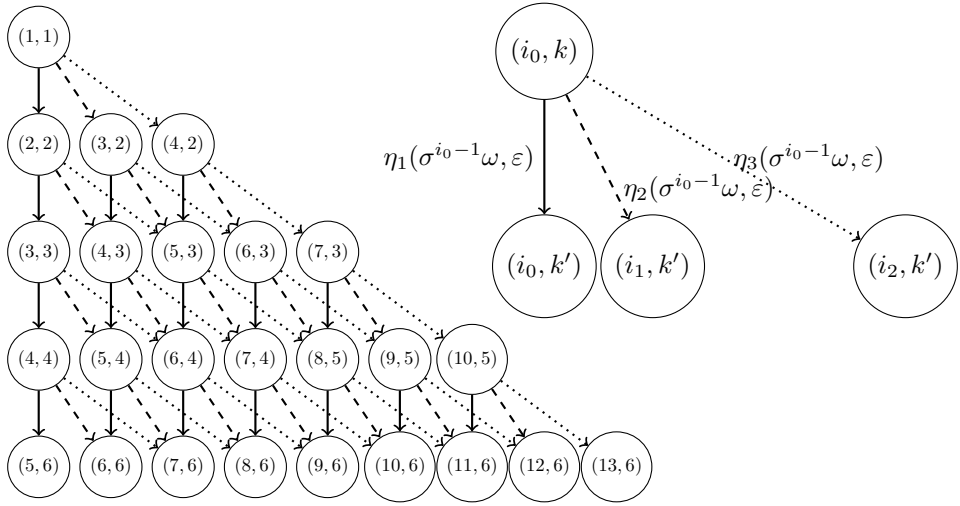

The last construction we shall require is a generalisation of the Hutchinson-Moran sum (see (\ref{HutchinsonMoranFormula})) to this infinite setting.
Let $\mathfrak{R}^{s}$, defined recursively, map matrices (or vectors) with entries being matrices over arrangements of words  into matrices (or vectors) with entries being matrices over real valued, non-negative functions, preserving the matrix (vector) structure.
\[
\epsilon_{0} \mapsto 1, \;\; \varnothing\mapsto0, \;\; \phi_{1}\mapsto c_{\phi_{1}}^{s}, \;\;  \phi_{1}\oplus\phi_{2}\mapsto c_{\phi_{1}}^{s}+c_{\phi_{2}}^{s},
\]
\[ \phi_{1}\odot\phi_{2}\mapsto c_{\phi_{1}}^{s}c_{\phi_{2}}^{s}=c_{\phi_{1}\phi_{2}}^{s},
\]
where $c_{\phi}$ is the contraction ratio of the similitude $S_{\phi}$.
We define $\mathbf{P}_{\epsilon}^{s}(\omega)=\mathfrak{R}^{s}(\mathbf{H}^{\epsilon}(\omega))$, that is the matrix consisting of rows 
\[
\mathbf{p}^{s}_{k}(\omega,\epsilon)=\{\mathbf{0},\dots,\mathbf{0},p^{s}_{1}(\omega,\epsilon),\dots,p_{l}^{s}(\omega,\epsilon),\mathbf{0},\dots\},
\] (c.f.\ $\widehat\eta^{i}(\omega,\epsilon)$) with 
\begin{equation}\label{HutchinsonMatrix}
p^{s}_{q}(\omega,\epsilon)=
\begin{pmatrix}
\sum_{e\in(\tensor*[_{1}]{E}{^{q}_{1}}(\omega,\epsilon))}c_{e}^{s}   &
  \sum_{e\in(\tensor*[_{1}]{E}{^{q}_{2}})(\omega,\epsilon)}c_{e}^{s}  &
    \hdots  &
     \sum_{e\in(\tensor*[_{1}]{E}{^{q}_{n}}(\omega,\epsilon))}c_{e}^{s}\\
\sum_{e\in(\tensor*[_{2}]{E}{_{1}^{q}}(\omega,\epsilon))}c_{e}^{s}  &
  \ddots  &
    &
      \vdots\\

\vdots &
  & 
   \ddots  &
     \vdots\\

\sum_{e\in(\tensor*[_{n}]{E}{^{q}_{1}}(\omega,\epsilon))}c_{e}^{s}   &
  \sum_{e\in(\tensor*[_{n}]{E}{^{q}_{2}}(\omega,\epsilon))}c_{e}^{s}  &
    \hdots  &
     \sum_{e\in(\tensor*[_{n}]{E}{^{q}_{n}}(\omega,\epsilon))}c_{e}^{s}

\end{pmatrix}.
\end{equation}

\subsection{Results for $1$-variable RGDS}
Having established the basic notation, in this section we collate all the important constructive lemmas and theorems. The proofs will be given in Section~\ref{proofsection}. We begin by stating that the norm $\vertiii{.}_{\sup}$ and seminorm $\vertiii{\mathds{1}.}$ expand almost surely at an exponential rate when multiplying the random matrices defined above; in other words the Lyapunov exponent exists.

\begin{lma}\label{convergenceLemma}
For $\mathbf{T}$ as above we have that
\begin{equation}\label{almostSureConvergenceLemmaEasy}
\lim_{k\to\infty}\vertiii{\mathbf{T}(\omega)\mathbf{T}(\sigma\omega)\dots \mathbf{ T}(\sigma^{k-2}\omega)\mathbf{T}(\sigma^{k-1}\omega)}_{\sup}^{1/k}=\alpha,
\end{equation}
where $\alpha=\inf_{k}\E^{\geo}(\vertiii{\mathbf{T}(\omega)\hdots\mathbf{T}(\sigma^{k-1}\omega)}_{\sup}^{1/k})$, for almost every $\omega\in\Omega$. If we use the seminorm defined in (\ref{infiniteSeminorm}),
almost surely,
\begin{equation}\label{almostSureConvergenceLemma}
\lim_{k\to\infty}\vertiii{\mathds{1}\mathbf{T}(\omega)\mathbf{T}(\sigma\omega)\hdots\mathbf{T}(\sigma^{k-2}\omega)\mathbf{T}(\sigma^{k-1}\omega)}^{1/k}=\beta,
\end{equation}
where $\beta\in[0,\infty)$ and $\mathds{1}=\{\mathbf{1},\mathbf{0},\mathbf{0},\hdots\}$. In particular,
\[\beta=\inf_{k}\vertiii{\mathds{1}\mathbf{T} 
(\omega)\hdots\mathbf{T}(\sigma^{k-1}\omega)}^{1/k}\text{ for a.e.\ $\omega$.}
\]
\end{lma}

We  apply this result to the our RGDS setting and prove that the Lyapunov exponent is independent of the row of the resulting matrix, assuming $\mathbf{\Gamma}^{\epsilon}$ satisfies Condition~\ref{strongConnected}. We define the norm of matrix products in our setting.

\begin{defn}\label{pressureDefinition}
Let $\epsilon>0$ and define 
\[
\Psi^{k}_{\omega}(s,\epsilon)=\vertiii{
 \mathds{1}\mathbf{P}^{s}_{\epsilon}(\omega)\mathbf{P}^{s}_{\epsilon}(\sigma\omega)\hdots \mathbf{P}^{s}_{\epsilon}(\sigma^{k-1}\omega)}^{1/k}\text{ and }
 \mathbf{\Psi}_{\omega}(s,\epsilon)=\lim_{k\to\infty}\Psi^{k}_{\omega}(s,\epsilon).
 \]
 We call $\mathbf{\Psi}_{\omega}(s,\epsilon)$ the \emph{$(s,\epsilon)$-pressure of realisation $\omega$}, if the limit exists, and we write $\mathbf{\Psi}(s,\epsilon)=\E^{\geo}\mathbf{\Psi}_{\omega}(s,\epsilon)$ for the \emph{$(s,\epsilon)$-pressure}.
\end{defn}

We note at this point that the notion of pressure is usually applied to $\log \mathbf{\Psi}$. However, in the $1$-variable setting it is more natural to talk about Lyapunov exponents and multiplicativity, rather than additivity, and we take the liberty to call these quantities pressures, rather than the more appropriate `exponential of pressures'.

\begin{lma}\label{rowSumsAllTheSame}
Assume $\mathbf{\Gamma}^{\epsilon}$, together with a non-trivial probability vector $\vec\pi$, is a non-trivial collection of graphs that satisfies Condition~\ref{strongConnected}. 
The exponential expansion rate of the norm of the matrix is identical to the expansion rate of each individual row sum. We have, almost surely, for every $v\in V$ and $\epsilon>0$
\[
\lim_{k\to \infty}\left[\sum_{w\in V}\left(\lVert\mathds{1}\mathbf{P}^{s}_{\epsilon}(\omega)\mathbf{P}^{s}_{\epsilon}(\sigma\omega)\hdots \mathbf{P}^{s}_{\epsilon}(\sigma^{k-1}\omega)\rVert_{s}\right)_{v,w}\right]^{1/k}=\mathbf{\Psi}(s,\epsilon).
\]
\end{lma}
%

\begin{lma}\label{almostSurePressureProperties}
For almost all $\omega$ we obtain $\mathbf{\Psi}(s,\epsilon)=\mathbf{\Psi}_{\omega}(s,\epsilon)$. Furthermore $\mathbf{\Psi}(s,\epsilon)$ is monotonically decreasing in $s$ and there exists a unique $s_{H,\epsilon}$ such that $\mathbf{\Psi}(s_{H,\epsilon},\epsilon)=1$.
\end{lma}

For $s=0$ the pressure function is counting the number of cylinders in the construction. However, as we are considering a lower approximation consisting solely of cylinders with diameter comparable to $\epsilon$ we can find the box counting dimension of $K_{v}(\omega)$ by a supermultiplicative argument. 

\begin{theo}\label{upperBox}
Almost surely the box counting dimension of $K_{v}(\omega)$ exists, is almost surely independent of $v\in V$, and given by
\begin{equation}\label{upperBoxEq}
\dim_{B}K_{v}(\omega)=\lim_{\delta\to 0}\frac{\log\mathbf{\Psi}(0,\delta)}{-\log\delta}=\sup_{\epsilon>0}\frac{\log\mathbf{\Psi}(0,\epsilon)}{-\log\epsilon}.
\end{equation}
\end{theo}
Using the construction given in Section~\ref{wordConstructionSection} we define the $\epsilon$-approximation to our attractor. Note that this is not an $\epsilon$-close set in the sense of Hausdorff distance, but rather an attractor which satisfies the Uniform Strong Separation Condition (USSC) and approximates the attractor from the `inside out'. Compare this to the approximation of GDA by suitably chosen IFSs, see Farkas~\cite{Farkas15}.
\begin{defn}
We say that a graph directed attractor satisfies the \emph{Uniform Strong Separation Condition} (USSC) if for every $v\in V$, $\Gamma_{k}\in\mathbf{\Gamma}$, $\omega\in\Omega$ and $e_{i},e_{j}\in \tensor*[_{v}]{E}{}(k)$ we have 
\[
\text{if }S_{e_{i}}(K_{v}(\omega))\cap S_{e_{j}}(K_{v}(\omega))\neq\varnothing\;\text{, then } \;e_{i}= e_{j}.
\]
\end{defn}
\begin{defn}
The \emph{$\epsilon$-approximation attractor} $K_{v,\epsilon}(\omega)$ of $K_{v}(\omega)$ is defined to be the unique compact set that is the limit of words in the $\epsilon$-stopping graph $\mathbf{\Gamma}^{\epsilon}$:
\begin{multline}
K_{v,\epsilon}(\omega)=\lim_{k\to\infty}\bigcap_{i=1}^{k}\bigcup_{\substack{\mathbf{e}\in \Xi^{i}_{\epsilon}(\omega)\\ \iota(\mathbf{e})=v}}S_\mathbf{e}(\Delta),
\\\text{where }
\Xi^{i}_{\epsilon}(\omega)=\bigoplus \mathds{1}_{\epsilon_{0}}\mathbf{H}^{\epsilon}(\omega)\mathbf{H}^{\epsilon}(\sigma\omega)\dots \mathbf{H}^{\epsilon}(\sigma^{i-1}\omega).\nonumber
\end{multline}
\end{defn}

These sets are easily seen to be subsets of $K_{v}(\omega)$.
\begin{lma}\label{approxAreSubsets}
For every $\epsilon>0$ and $\omega\in\Omega$ we have $K_{v,\epsilon}(\omega)\subseteq K_{v}(\omega)$. If $K_{v}(\omega)$ satisfies the USSC, then $K_{v,\epsilon}(\omega)= K_{v}(\omega)$.
\end{lma}
\begin{proof}
Note that points in the attractor of $K_{v,\epsilon}(\omega)$ have (unique) coding given by edges of graphs $\Gamma^{\epsilon}_{i}$ in $E(\omega,\epsilon)$. To prove the first claim we observe that for every symbol $e_{i}$ in the coding of $x=(e_{1},e_{2},\dots)\in K_{v,\epsilon}(\omega)$ we have an equivalent path travelling through $\mathbf{\Gamma}$. Starting at the first edge we have
$e_{1}\in E^{q_{1}}(\omega,\epsilon)$ for some $q_{1}$. This means that $e_{1}=\hat e^{1}_{1} \hat e^{1}_{2}\dots \hat e^{1}_{q_{1}}$ for $\hat e^{1}_{j}\in E(\omega_{j})$ such that $\tau (\hat e_{j})=\iota (\hat e_{j+1})$. 
Furthermore $e_{2}\in E^{q_{2}}(\sigma^{q_{1}}\omega,\epsilon)$ and so $e_{2}=\hat e^{2}_{1} \hat e^{2}_{2}\dots \hat e^{2}_{q_{2}}$ for $\hat e^{2}_{j}\in E((\sigma^{q_{1}}\omega)_{j})$ for a similarly linked sequence of edges.
Inductively we can replace every edge in $x$ by a finite path in the appropriate manner, giving a coding of a point in $K_{v}(\omega)$ and thus $K_{v,\epsilon}(\omega)\subseteq K_{v}(\omega)$.

Now assume that the maps of $\mathbf{\Gamma}$ satisfy the USSC; for all $v\in V$ and $i\in\Lambda$, every $e_{1},e_{2}\in \tensor*[_{v}]{E}{}(i)$ satisfy $S_{e_{1}}(K_{\tau(e_{1})}(\omega))\cap S_{e_{2}}(K_{\tau(e_{2})}(\omega))=\varnothing$. But then for all $j\in\Lambda$,  $e_{11}\in\tensor*[_{w_{1}}]{E}{}(j)$ and $e_{21}\in \tensor*[_{w_{2}}]{E}{}(j)$, where $w_{1}=\tau(e_{1})$ and $w_{2}=\tau(e_{2})$, we have $S_{e_{1}e_{11}}(K_{\tau(e_{11})}(\omega))\cap S_{e_{2}e_{21}}(K_{\tau(e_{21})}(\omega)) =\varnothing $. Inductively none of the compositions overlap. But this means that every path traversing through $\mathbf{\Gamma}$ must also have an equivalent path traversing through $\mathbf{\Gamma}^{\epsilon}$ as no paths get deleted due to the non-existent overlaps. Hence, assuming the USSC, $K_{v}(\omega)\subseteq K_{v,\epsilon}(\omega)$.
\end{proof}

Having established the almost sure box counting dimension we now consider the Hausdorff dimensions of our approximation sets. These are given by the unique $s$ such that the pressure defined in~(\ref{pressureDefinition}) equals $1$ and form a lower bound of the Hausdorff dimension of $K_{v}(\omega)$.

\begin{theo}\label{lowerHausdorff}
For all $\epsilon>0$ the almost sure Hausdorff dimension of $K_{v,\epsilon}(\omega)$ is independent of $v\in V$ and 
\[
\dim_{H}K_{v,\epsilon}(\omega) = s_{H,\epsilon}\text{ where }\mathbf{\Psi}(s_{H,\epsilon},\epsilon)=1,
\]
where $s_{H,\epsilon}$ is given by Lemma~\ref{almostSurePressureProperties}.
\end{theo}

We get the following important corollary to Lemma~\ref{approxAreSubsets} and Theorem~\ref{lowerHausdorff}.
\begin{cor}
The Hausdorff dimension of the attractor of the $1$-variable self-similar RGDS is, almost surely, bounded below by $s_{H,\epsilon}$ for all $\epsilon>0$
\[
\dim_{H}K_{v}(\omega)\geq \dim_{H}K_{v,\epsilon}(\omega)=s_{H,\epsilon}.
\]
\end{cor}

Our main result is the almost sure equality of Hausdorff, box-counting and therefore also packing dimension, of $K_{v}(\omega)$ for all $v\in V$.

\begin{theo}[Main Theorem]\label{dimensionEquality}
Let $\mathbf{\Gamma}$ be a non-trivial, stochastically strongly connected collection of graphs with associated self-similar attractors $\{K_v\}_{v\in V}$. Then $s_{H,\epsilon}\to s_{B}$ as $\epsilon\to 0$, where 
\[
s_{B}=\lim_{\epsilon\to0}\frac{\log\mathbf{\Psi}(0,\epsilon)}{-\log\epsilon}
\]
 and hence, almost surely, 
\[
\dim_{H}K_{v}(\omega)=\dim_{P}K_{v}(\omega)=\dim_{B}K_{v}(\omega)=s_{B},
\]
where $s_{B}$ is independent of $v$.
\end{theo}

If the attractor of $\mathbf{\Gamma}$ satisfies the USSC we can in addition give an easy description of the almost sure dimension of the attractor.
\begin{cor}\label{separatedCorollary}
Assume the USSC is satisfied, then $s_{H,\epsilon}=s_{B}$ for all $\epsilon>0$, and, almost surely,
\begin{equation}\label{USSCCorFormula}
\dim_{H}K_{v}(\omega)=\dim_{B}K_{v}(\omega)=s_{O}\text{, where }\lim_{k\to\infty}
\lVert p_{1}^{s_{O}}(\omega,1)\dots p_{1}^{s_{O}}(\sigma^{k-1}\omega,1)\rVert_{1}^{1/k}=1.
\end{equation}
Equivalently, $s_{O}$ is the unique non-negative real satisfying 
\[\inf_{k}(\E^{\geo}\lVert p_{1}^{s_{O}}(\omega,1)\dots p_{1}^{s_{O}}(\sigma^{k-1}\omega,1)\rVert_{1})^{1/k}=1.
\]
\end{cor}\label{VvarandOther}
Because $V$-variable self-similar sets are $1$-variable RGDS self-similar and under the assumption that $\mathbf{\Gamma}$ satisfies the USSC, Corollary~\ref{separatedCorollary} reduces to the results in Barnsley et al.~\cite{Barnsley12}. Additionally we get the following new result:
\begin{cor}
Let $F(\omega)$ be the attractor of a $V$-variable random iterated function system. Irrespective of overlaps, almost surely,
\[
\dim_{H}F(\omega)=\overline\dim_{B}F(\omega)=\dim_{B}F(\omega).
\]
\end{cor}
This follows since the construction of a $V$-variable set relies on a vector of sets of dimension $V$. Associating a vertex to each of these sets we can chose graphs appropriately. 

However, in contrast to all other dimensions, the Assouad dimension `maximises' the dimension. This phenomenon has been observed in many different settings, which is not surprising as the Assouad dimension `searches' for the relatively most complex part in the attractor and the random construction allows a very complex pattern to arise on many levels with probability one, even though these events get `ignored' by the averaging behaviour of Hausdorff and box-counting dimension.

\begin{defn}\label{jointSpectralDef}
Let $\mathbf{\Gamma}$ be as above. We define the \emph{$\epsilon$-joint spectral radius} by
\[
\mathfrak{P}(\epsilon)=\lim_{k\to\infty}\left(\sup_{\omega\in\Omega}\left\{\vertiii{\mathds{1}\mathbf{P}_{\epsilon}^{0}(\omega)\mathbf{P}_{\epsilon}^{0}(\sigma^{1}\omega)\hdots\mathbf{P}^{0}_{\epsilon}(\sigma^{k-1}\omega)}\right\}\right)^{1/k}.
\]
\end{defn}
We note that the spectral radius coincides for almost every $\zeta\in\Omega$ with the limit in~(\ref{almostSureConvergenceLemmaEasy}):
\begin{equation}\label{spectralEquivalence}
\mathfrak{P}(\epsilon)=\alpha=\lim_{k\to\infty}\vertiii{\mathbf{P}_{\epsilon}^{0}(\zeta)\hdots\mathbf{P}_{\epsilon}^{0}(\sigma^{k-1}\zeta)}_{\sup}^{1/k}.
\end{equation}
We demonstrate this in the proof of Theorem~\ref{AssouadTheo}.

\begin{theo}\label{AssouadTheo}
Assume $K_{v}(\omega)\subset \R^{d}$ is not contained in any $d-1$-dimensional hyperplane for all $v\in V$ and almost all $\omega\in\Omega$. 
Irrespective of separation conditions, almost surely,
\begin{equation}\label{Assouad1VarEq}
\dim_{A}K_{v}(\omega)\geq\min\left\{d,\sup_{\epsilon>0}\frac{\log\mathfrak{P}(\epsilon)}{-\log\epsilon}\right\}.
\end{equation}
Further, the USSC implies equality in~(\ref{Assouad1VarEq}).
\end{theo}

\section{$\infty$-variable Random Graph Directed Systems}\label{infVarSect}
In this section we introduce and provide results for the $\infty$-variable construction. In a similar fashion to Section~\ref{resultsSection} we start by giving a description of the model and then state the results. For the $\infty$-variable construction many proofs turn out to be simpler and to save space we shall give less detail in some of the proofs as they follow from standard arguments.
\subsection{Notation and Model}
The $\infty$-variable model, sometimes called random recursive or $V$-variable for $V\to\infty$, is a very intuitive model that is usually defined in a recursive manner (see \cite{Falconer86} and \cite{Graf87}).
A more standard and useful notation would be adapting the notation of random code trees. For an overview of that notation we refer the reader to J\"arvenp\"a\"a, et al.~\cite{Jarvenpaa14} who studied a different random model with a `neck structure'. However, to keep notation consistent we will describe the random recursive construction within our framework of arrangements of words. Note that, unlike the $1$-variable construction, the $\infty$-variable construction overlaps considerably with the notion of random graph directed attractors, considered in Olsen~\cite{Olsen94}, and some of the results here follow directly from the ones in aforementioned book. 

As in Section~\ref{resultsSection} we are given a collection of graphs $\mathbf{\Gamma}$ with associated non-trivial probability vector $\vec\pi$. We further assume that all the maps given by the edges of the $\Gamma_{i}$ are contracting similitudes  and that all conditions in Definition~\ref{graphDefs} are satisfied. However, we can generalise the results to include percolation by adapting Condition~\ref{nontrivialgraph}.
\begin{defn}\label{replacementCondition}
Let $\mathbf{\Gamma}=\{\Gamma_{i}\}_{i\in\Lambda}$ be a finite collection of graphs, sharing the same vertex set $V$.
We say that the collection $\mathbf{\Gamma}$ is a \emph{non-trivial surviving collection of graphs} if for every $v\in V$ we have $\E(\#\tensor*[_{v}]{E(\omega_{1})}{})>1$: there exists positive probability that the resulting $\infty$-variable RGDS coding does not consist of only $\varnothing$, and there exist $i,j\in\Lambda$ and $e_{1}\in\Gamma(i)$ and $e_{2}\in \Gamma(j)$ such that $S_{e_{1}}\neq S_{e_{2}}$.
\end{defn}

\begin{defn}
For $v\in V$ let $\mathbf{F}^{0}_{v}$ be a vector of length $n=\lvert V\rvert$ defined by
\[
(\mathbf{F}_{v}^{0})_{i}=\begin{cases}
\epsilon_{0}	&	\text{ if }i=v,\\
\varnothing	&	\text{ otherwise.}
\end{cases}
\]
We then define inductively,
\[
(\mathbf{F}^{k+1}_{v})_{i}=\bigoplus_{j=1}^{n}\;\bigoplus_{\mathbf{w}\in(\mathbf{F}^{k}_{v})_{j}}\;\bigoplus_{e\in \tensor[_{j}]{E}{_{i}}(\xi_\mathbf{w})}\mathbf{w}\odot e,
\]
where $\xi_\mathbf{w}$ is the random variable given by $\Prob(\xi_\mathbf{w}=i)=\pi_{i}$ for $i\in\Lambda$ and independent of $\mathbf{w}$. 

The $\infty$-variable RGDS coding is then given by $\mathbf{F}_{v}=\lim_{k\to\infty}\mathbf{F}_{v}^{k}$ and we define the attractor $F_{v}$ of the $\infty$-variable Random Graph Directed System to be the projection of our coding set:
\[
F_{v}=\bigcap_{k=1}^{\infty}\bigcup_{\mathbf{w}\in \mathbf{F}_{v}^{k}}S_{w_{1}}\circ S_{w_{2}}\circ\dots\circ S_{w_{k}}(\Delta)
\]
\end{defn}

%
%

Given a collection of graphs satisfying Conditions~\ref{strongConnected}, \ref{contractingSSRGDS} and \ref{replacementCondition} that do not necessarily satisfy the USSC we obtain an analogous definition of the $\epsilon$-approximation.

Let $\mathcal{Q}$ be the space of all possible realisations of the random recursive process, $\mathcal{Q}$ is a labeled tree encoding which graph $\Gamma(i)$ was chosen at each node in the construction of the tree. This means that for every word $\mathbf{w}\in\bigcup_{i=0}^{\infty}\mathbf{F}^{i}_{v}$ we associate an $i\in\Lambda$ and for $\mathbf{w}\in \mathbf{F}^{k}_{v}$ the set of infinite words $\mathbf{x}$ satisfying $\mathbf{x}\wedge \mathbf{w}=\mathbf{w}$ for the subbranches at node $\mathbf{w}$. By the same argument as in Section~\ref{stoppingGraphs} for every fixed $\epsilon>0$ there exists a finite constant $k_{\max}(\epsilon)$ such that for all $\mathbf{w}\in \mathbf{F}^{k_{\max}(\epsilon)}_{v}$ we have $\lvert S_\mathbf{w}(\Delta)\rvert\leq\epsilon$ for all realisations $q\in\mathcal{Q}$. Now $F_{v}$ is a function mapping realisations to compact sets, depending solely on the random variable $q\in\mathcal{Q}$ (picked according to the Borel probability measure induced by $\vec\pi$) but, in general, we shall ignore the $q$ in the notation of $F_{v}(q)$.

\begin{defn}
Let $\mathbf{\Gamma}$ satisfy the conditions in Definition~\ref{graphDefs}. Let $\mathcal{Q}$ be the space of all possible realisations of the random recursive process, we define the set of edges (words) of length $j$ for realisation $q$ to be $\mathbf{F}^{k}_{v}(q)$ and the $\epsilon$-stopping set of edge sets to be
\[
E^{*}_{v}(q,\epsilon)=\left\{ \mathbf{e}\in\bigcup_{i=1}^{k_{\max}(\epsilon)} \mathbf{F}^{i}_{v}(q)  \mid  c_\mathbf{e}\leq \epsilon \text{ but }c_{\mathbf{e}^{\ddagger}}>\epsilon \right\}.
\]
Again let the set of all possible subsets such that images under $S$ are pairwise disjoint be 
\[
\mathcal{E}(q,\epsilon)=\{U\subseteq E^{*}_{v}(q,\epsilon)\mid \forall \mathbf{e},\mathbf{f}\in U\text{ we have }S_\mathbf{e}\cap S_\mathbf{f}=\varnothing\}.
\]
Consider the element of maximal cardinality (choosing arbitrarily if there is more than one) $E_{v}(q,\epsilon)\in\mathcal{E}(q,\epsilon)$. As $E_{v}(q,\epsilon)$ only depends, at most, on the first $k_{\max}(\epsilon)$ entries, the set $\{E_{v}(q,\epsilon)\}_{q\in\mathcal{Q}}$ is finite and we write 
\[
\mathbf{\Gamma}^{\epsilon}=\{\Gamma^{\epsilon}(q)\}_{q\in\mathcal{Q}}=\{(V,E_{v}(q,\epsilon))\}_{q\in\mathcal{Q}}
\]
for the $\epsilon$-stopping graph.
\end{defn}
As $\mathbf{\Gamma}^{\epsilon}$ is finite we will set up a new code space for each of the graphs $\Gamma^{\epsilon}(q)$ that we will index by $\Lambda_{\epsilon}$. Similarly there exists positive probability of picking graph $\Gamma^{\epsilon}(\lambda)$ for $\lambda\in\Lambda_{\epsilon}$.
Unlike the $1$-variable case, the choice of graph $\mathbf{\Gamma}$ is independent for each node, a property which transfers to the setting of the $\epsilon$-stopping graph.  
\begin{lma}\label{sameInfiniteProcess}
The random recursive algorithm that generates the attractor of the $\epsilon$-stopping graphs $\mathbf{\Gamma}^{\epsilon}$  is identical to the process that generates the attractor of the RGDS $\mathbf{\Gamma}$. 
Note that for $t\geq1$ the identity $\mathbf{\Gamma}=\mathbf{\Gamma}^{t}$ holds and we trivially have that the attractor of the RGDS $\mathbf{\Gamma}^{\epsilon}$ is a subset of the attractor of $\mathbf{\Gamma}$, with equality holding if the attractor of  $\mathbf{\Gamma}$ satisfies the USSC.
\end{lma}
We  omit a detailed proof as both processes can easily seen to be $\infty$-variable RGDS. Now let $\mathbf{K}^{\epsilon}(q)$ be the matrix consisting of arrangements of words related to $\Gamma^{\epsilon}(q)$. Let $\tensor[_{v}]{E}{_{w}}(\Gamma(q))$ be the collection of edges $e$ of $\Gamma(q)$ so that $\iota(e)=v$ and $\tau(e)=w$, and define
\[
\mathbf{K}^{\epsilon}(q)=\begin{pmatrix}
\bigoplus_{e\in \tensor*[_{1}]{E}{_{1}}(\Gamma^{\epsilon}(q))}e	&	\dots		&	\bigoplus_{e\in \tensor*[_{1}]{E}{_{n}}(\Gamma^{\epsilon}(q))}e \\
\vdots		&		\ddots		&\vdots		\\
\bigoplus_{e\in \tensor*[_{n}]{E}{_{1}}(\Gamma^{\epsilon}(q))}e	&	\dots		&	\bigoplus_{e\in \tensor*[_{n}]{E}{_{n}}(\Gamma^{\epsilon}(q))}e
\end{pmatrix}.
\]

\begin{theo}\label{infiniteSameDimension}
Let $\mathbf{\Gamma}$ be a finite collection of graphs satisfying Conditions~~\ref{strongConnected}, \ref{contractingSSRGDS} and \ref{replacementCondition} with associated non-trivial probability vector $\vec\pi$. Let $F_{v}$ be the attractor of the random recursive construction, then almost surely the Hausdorff and the upper box counting dimension agree and thus,
\[
\dim_{H}F_{v}=\dim_{P}F_{v}=\dim_{B}F_{v}.
\]
\end{theo}

We end this section by stating the Assouad dimension of this construction.
\begin{theo}\label{infiniteAssouadDimension}
Irrespective of overlaps and conditioned on $F_{v}\neq\varnothing$, the Assouad dimension of $F_{v}$ is a.s.\ bounded below by 
\begin{equation}\label{infiniteAssouadEq}
\dim_{A}F_{v}\geq\min\left\{d,\; \sup_{\epsilon>0}\max_{q\in\mathcal{Q}}\frac{\log\rho(\mathfrak{R}^{0}\mathbf{K}^{\epsilon}(q))}{-\log\epsilon}\right\}.
\end{equation}
where $\rho$ is the spectral radius of a matrix.
If the USSC is satisfied, then equality holds in~(\ref{infiniteAssouadEq}) almost surely.
\end{theo}

\section{Proofs}\label{proofsection}
\subsection{Proof of Lemma \ref{convergenceLemma}}
First we prove the convergence in equation~(\ref{almostSureConvergenceLemmaEasy}). Let $n,m\in\N_{0}$, $n<m$ and define the random variable $Y_{n,m}$ as
\[
Y_{n,m}(\omega)=\log\vertiii{\mathbf{T}(\sigma^{n}\omega)\mathbf{T}(\sigma^{n+1}\omega)\hdots\mathbf{T}(\sigma^{m-1}\omega)}_{\sup}
\]
Note that, as the row norm is submultiplicative,
\begin{align}
Y_{0,n+m}(\omega)&=\log\vertiii{\mathbf{T}(\omega) \hdots\mathbf{T}(\sigma^{n-1}\omega)\mathbf{T}(\sigma^{n}\omega)\hdots\mathbf{T}(\sigma^{n+m-1}\omega)}_{\sup}  \nonumber\\
&\leq\log\left(\vertiii{\mathbf{T}(\omega)\dots\mathbf{T}(\sigma^{n-1}\omega)}_{\sup}\vertiii{\mathbf{T}(\sigma^{n}\omega)\dots\mathbf{T}(\sigma^{n+m-1}\omega)}_{\sup}\right)\nonumber\\
&=\log\vertiii{\mathbf{T}(\omega)\hdots\mathbf{T}(\sigma^{n-1}\omega)}_{\sup}+\log\vertiii{\mathbf{T}(\sigma^{n}\omega)\hdots \mathbf{T}(\sigma^{n+m-1}\omega)}_{\sup}\nonumber\\
&=Y_{0,n}(\omega)+Y_{n,m}(\omega).	\nonumber
\end{align}
As $\mu$ is an ergodic probability measure it follows from Kingman's subadditive ergodic theorem that almost surely 
\[
\lim_{k\to\infty}\frac{Y_{0,k}}{k}=\inf_{k}\E \frac{Y_{0,k}}{k}=\inf_{k}\E\log\vertiii{\mathbf{T}(\sigma^{k-1}\omega)\hdots\mathbf{T}(\omega)}_{\sup}^{1/k}=\log\alpha,
\]
giving the required result.\qed

\vspace{0.5cm}
The second part is made slightly more difficult because of the interdependence between the steps. We will show stochastic quasi-subadditivity, bounding the subadditive defects, and make use of the following variant of Kingman's subadditive ergodic theorem, see also Kingman~\cite{Kingman73}.

\begin{prop}[Derriennic,~\cite{Derriennic83}]\label{KingmanProp}
Let $X_{m}(\omega)$ be a (measurable) random variable on a probability space $(\Omega,\mu)$ and let $T$ be a measurable, measure preserving map. If the expectation of the subadditive defects is bounded by a sequence of reals numbers $(c_{m})$, i.e.\ for all $n,m\geq1$,
\[
\E(X_{n+m}(\omega)-X_{n}(\omega)-X_{m}(T^{n}\omega))^{+}\leq c_{m},
\]
where $c_{m}$ satisfies $\lim_{k}c_{k}/k\to0$, and $\E\inf_{k}X_{k}/k>-\infty$, then $X_n/n$ converges in $\mathcal{L}^1$ to some random variable taking values in $\R$. If further,
\[
X_{n+m}(\omega)-X_{n}(\omega)-X_{m}(T^{n}\omega)\leq Y_m(T^n \omega) \quad\text{(almost surely)}
\]
for some stochastic process $(Y_m)_m$ satisfying $\sup_m \E (Y_m)<\infty$, then $X_{n}/n$ converges almost surely to some random variable $\eta\in(-\infty,\infty)$.
\end{prop}

If $T$ is ergodic with respect to $\Prob$, then $\eta$ is constant for almost every $\omega$ as 
\[
\{\omega\in\Omega \mid \liminf_{n\to\infty}X_{n}(\omega)/n>z\}=\{\omega\in\Omega \mid \liminf_{n\to\infty}X_{n}(T\omega)/n>z\}.
\] Since for $p>1$, the $p$-th moment satisfies $((c_{k}^{+})^{p})/k\to 0$ the limit necessarily coincides with $\lim_{k}\E (X_{k})/k=\inf_{k}\E (X_{k})/k$.

Writing $\mathbf{u}_{k}(\omega)=\mathbf{T}(\omega)\dots\mathbf{T}(\sigma^{k-1}\omega)$
 the term $\mathds{1}\mathbf{u}_{k}(\omega)$ is a matrix-valued vector with at most $lk$ positive entries, all appearing in the first $lk$ rows, where $l\geq1$ as in Section~\ref{infiniteMatrixSection}.
We have
\begin{align}
\vertiii{\mathds{1}\mathbf{u}_{n+m}(\omega)}
&=\vertiii{\mathds{1}\mathbf{u}_{n}(\omega)\mathbf{u}_{m}(\sigma^{n}\omega)}\nonumber\\
&=\left\lVert\,\lVert \mathds{1}\mathbf{u}_{n}(\omega)\mathbf{u}_{m}(\sigma^{n}\omega) \rVert_{s}\,\right\rVert_{\text{row}}\nonumber\\
&=\left\lVert\sum_{j=0}^{nl-1} \left(\mathds{1}\mathbf{u}_{n}(\omega)\right)_{j}   \left\lVert \mathds{1}\mathbf{u}_{m}(\sigma^{n+j}\omega) \right\rVert_{s}  \right\rVert_{\text{row}}\nonumber\\
&\leq\sum_{j=0}^{nl-1}\left\lVert \left(\mathds{1}\mathbf{u}_{n}(\omega)\right)_{j}   \left\lVert \mathds{1}\mathbf{u}_{m}(\sigma^{n+j}\omega) \right\rVert_{s}\right\rVert_{\text{row}} \text{ by subadditivity of norms,}\nonumber\\
&\leq\sum_{j=0}^{nl-1} \left\lVert \left(\mathds{1}\mathbf{u}_{n}(\omega)\right)_{j}\right\rVert_{\text{row}}  \vertiii{\mathds{1}\mathbf{u}_{m}(\sigma^{n+j}\omega)}   \nonumber
\\&\hspace{4cm}\text{ by submultiplicativity of the row norm,}\nonumber\\
&\leq nl	\left\lVert \left(\mathds{1}\mathbf{u}_{n}(\omega)\right)_{j_{\max}(n,m,\omega)}\right\rVert_{\text{row}}  \vertiii{\mathds{1}\mathbf{u}_{m}(\sigma^{n+j_{\max}(n,m,\omega)}\omega)}	\label{nonexpectationsubmult}\\
&\hspace{4cm}\text{ for $j_{\max}$ maximising the sum,}\nonumber\\
&\leq cnl \vertiii{\mathds{1}\mathbf{u}_{n}(\omega)}\vertiii{\mathds{1}\mathbf{u}_{m}(\sigma^{n}\omega)}
\end{align}
The last inequality holds for some sufficiently large $c>0$ upon noting that for large $n,m$ the additional shift $j_{\max}$ becomes insignificant as the difference in growth is captured by the `overestimate' of the first term. Therefore we have quasi-subadditivity and by symmetry
\[
\vertiii{\mathds{1}\mathbf{u}_{n+m}(\omega)}\leq c m \vertiii{\mathds{1}\mathbf{u}_{n}(\omega)}\vertiii{\mathds{1}\mathbf{u}_{m}(\sigma^{n}\omega)},
\]
for some $c>0$.
Considering $\log\vertiii{\mathds{1}\mathbf{u}_{n}(\omega)}$ as a random variable, the subadditive defect becomes
\[
c_{m}=\log\vertiii{\mathds{1}\mathbf{u}_{n+m}(\omega)}-\log\vertiii{\mathds{1}\mathbf{u}_{n}(\omega)}-\log\vertiii{\mathds{1}\mathbf{u}_{m}(\sigma^{n}\omega)}\leq \log cm.
\]
Clearly $\E(\log cm)^{+}=\log cm$ and $c_{m}/m\to 0$. Since $\sigma$ is an (invariant) ergodic transformation with respect to $\mu$, applying Proposition~\ref{KingmanProp} finishes the proof. \qed

\subsection{Proof of Lemma \ref{rowSumsAllTheSame}}
The boundedness of the entries in the matrix entries of $\mathds{1}\mathbf{u}_{k}(\omega)$, combined with the linear growth of the number of positive entries of the vector, implies that for some constant $c>0$,
\[
\max_{j}\left\lVert(\mathds{1}\mathbf{u}_{k}(\omega))_{j}\right\rVert_{\text{row}}\leq  \vertiii{\mathds{1}\mathbf{u}_{k}(\omega)}\leq ck\max_{j}\left\lVert(\mathds{1}\mathbf{u}_{k}(\omega))_{j}\right\rVert_{\text{row}}.
\] 
Therefore the value of both terms increase at the same exponential rate.
In addition, the $j^{k}_{\max}$ maximising the norm cannot move arbitrarily with increasing $k$. First it must be increasing monotonically, although not necessarily strictly so. 
But the value can also not jump unboundedly, as the matrices that the matrix with maximal absolute norm is multiplied with have bounded entries as well. 
Even though we will not prove it here, it can be shown that almost surely $j_{\max}^{k}/(lk)\to \rho$ as $k\to\infty$ for some $\rho\in[0,1]$ dependent only on $\mathbf{\Gamma^{\epsilon}}$ and $\vec{\pi}$.
Let $R_{v}(k)$ be the row sum for row $v$ in the maximal matrix at multiplication step $k$ and $R_{v}^{T}(k)$ be the total of that row over all matrices. That is
\[
R_{v}(k)=\sum_{i=1}^{n}\left[(\mathds{1}\mathbf{u}_{k}(\omega))_{j_{\max}(n,m,\omega)}\right]_{v,i}
\;\;\;\text{ and }\;\;\;
R_{v}^{T}(k)=\sum_{j=1}^{\infty}\sum_{i=1}^{n}\left[(\mathds{1}\mathbf{u}_{k}(\omega))_{j}\right]_{v,i}.
\]
Furthermore let $R_{\max}(k)=\max_{v\in V}R_{v}(k)$. One immediately has on a full measure set 
\[
\left\lvert R_{\max}(k)^{1/k}- \vertiii{\mathds{1}\mathbf{u}_{k}(\omega)}^{1/k}\right\rvert\to0 \text{ as }k\to\infty,
\] 
so proving Lemma~\ref{rowSumsAllTheSame} can be achieved by showing $R_{v}(k)\asymp R_{\max}(k)$ holds almost surely for all $v\in V$. The upper bound $R_{v}(k)\leq R_{\max}(k)$ is trivial. 

For the lower bound, since $\mathbf{\Gamma}$ is stochastically strongly connected, i.e.\ satisfies Condition~\ref{strongConnected}, we can construct a finite word $\omega^{r}\in\Lambda^{*}$ that links all vertices, starting at $v=v_{1}$. 
That is $\omega^{r}=\omega^{v_{1},v_{2}}\omega^{v_{2},v_{3}}\dots\omega^{v_{n},v_{1}}\omega^{v_{1},v_{2}}\dots\omega^{v_{n-1},v_{n}}$. Clearly $\mu([\omega^{r}])>0$. 
Consider now the maximal element in the multiplication of $\mathbf{u}_{qk}(\omega)=\mathbf{u}_{k}(\omega)\dots\mathbf{u}_{k}(\sigma^{(q-1)k}\omega)$, that is $j_{\max}(qk,k,\omega)$.
There exists a random variable, the \emph{holding time} $H(i)$, that gives the number of multiplication steps $q$ between the $i-1$ and $i$th time such that $\omega^{r}$ is applied to that element. We have $\sigma^{qk+j_{\max}(qk, k,\omega)}(\omega)=\omega^{r}$. We can without loss of generality assume that $H(i)$ are i.i.d.\ random variables with finite expectation $\E H(i)<\infty$.
Let $W(k)$ be the waiting time for the $k$th jump, $W(k)=\sum_{i=0}^{k-1}H(i)$ and define $N_{k}$ to be the unique random integer such that 
\[
W(N_{k})\leq k < W(N_{k}+1).
\]

There exists a uniform constant $\underline\lambda>0$ such that, for all $v\in V$, 
\[
R_{v}(W(N_{k})+\lvert\omega^{r}\rvert)\geq\underline\lambda R_{\max}(W(N_{k})).
\]
Since this holds for all $k$ we can furthermore find a lower bound to the value of $R_{v}$ between occurrences of $\omega^{r}$ by considering the time it takes between occurrences. Condition~\ref{nontrivialgraph} implies non-extinction and there exists contraction rate $\underline\gamma>0$, such that for $k$ and $N_{k}$ as above we have
\[
\liminf_{k\to\infty}R_{v}^{T}(k)^{1/k}\geq \liminf_{k\to\infty}(\underline\lambda R_{\max}(W(N_{k}))\underline\gamma^{H(k)})^{1/k}
\geq\liminf_{k\to\infty}(\beta-\epsilon)^{W(N_{k})/k}\underline\gamma^{H(k)/k}
\]
where the last inequality holds on a set of measure $1$ for every $\epsilon>0$.
But we also have that 
\begin{align}
W(N_{k})/k\leq 1 < W(N_{k}+1)/k\nonumber
\end{align}
and as $W(N_{k})/k<1$ and $W(N_{k}+1)/k=W(N_{k})/k+H(N_{k}+1)/k$ we have by the law of large numbers that almost surely $W(N_{k})/k\to1$ and $H(k)/k\to0$, and hence on a set of measure 1, 
\[
\liminf_{k\to\infty}R_{v}^{T}(k)^{1/k}\geq(\beta-\epsilon)
\] for every $\epsilon$ and $v$. Noting that $R_{v}^{T}(k)\asymp R_{v}(k)$ completes the proof. 
\qed

\subsection{Proof of Lemma \ref{almostSurePressureProperties}}
The almost sure convergence of $\mathbf{\Psi}(s,\epsilon)$ follows directly from Lemma~\ref{convergenceLemma} and we now show that $\mathbf{\Psi}_\omega(s,\epsilon)$ is monotonically decreasing in $s$ and continuous for almost all $\omega\in\Omega$.
Consider an arbitrary Hutchinson-Moran sum that arises in the Hutchinson-Moran-like matrix in~(\ref{HutchinsonMatrix}),
\[
\sum_{e\in(\tensor*[_{i}]{E}{^{q}_{j}}(\omega,\epsilon))}c_{e}^{s}.
\]
We immediately get 
\begin{equation}\label{HutchinMatrixIneq}
\sum_{e\in(\tensor*[_{i}]{E}{^{q}_{j}}(\omega,\epsilon))}c_{e}^{s+\delta}\leq\overline{\gamma}^{\delta}_{q}\sum_{e\in(\tensor*[_{i}]{E}{^{q}_{j}}(\omega,\epsilon))}c_{e}^{s}\text{, where }
\overline\gamma_{q}(\omega)=\max_{\substack{i,j\in\{1,\hdots,n\} \\ \mathbf{e}\in(\tensor*[_{i}]{E}{^{q}_{j}}(\omega,\epsilon))}} c_\mathbf{e}.
\end{equation}
For $\epsilon>0$ there are only finitely many different $p_q^s(\omega,\epsilon)$ and $\mathbf{p}^s(\omega)$, see the discussion of Lemma~\ref{lma:stoppingrepl}. 
Thus we can find 
\begin{equation}
\overline\gamma=\max_{\substack{q\in\{1,\hdots,l\}\\ \omega\in\Omega}}\overline\gamma_{q}(\omega),
\end{equation}
where $0<\overline\gamma<1$. Similarly we can find the minimal such contraction $0<\underline\gamma\leq\overline\gamma<1$. Combining this with~(\ref{HutchinMatrixIneq}) we surely deduce, in turn,
\begin{align}
\underline\gamma^{\delta}p_q^{s}(\omega,\epsilon) &\leq p_q^{s+\delta}(\omega,\epsilon)\leq \overline\gamma^{\delta}p_q^{s}(\omega,\epsilon),\nonumber
\\
\underline\gamma^{\delta}\mathbf{p}^{s}(\omega,\epsilon) &\leq \mathbf{p}^{s+\delta}(\omega,\epsilon)\leq \overline\gamma^{\delta}\mathbf{p}^{s}(\omega,\epsilon),\nonumber
\\\label{pressureMatrixinEq}
\underline\gamma^{\delta}\mathbf{P}_\epsilon^s(\omega)&\leq\mathbf{P}_\epsilon^{s+\delta}(\omega)\leq\overline\gamma^{\delta}\mathbf{P}_\epsilon^{s}(\omega),
\end{align}
where $\leq$ is taken to be entry-wise, i.e.\ for matrices $M\leq N$ if and only if $M_{i,j}\leq N_{i,j}$ for all $i,j$.
Using~(\ref{pressureMatrixinEq}) we can bound the $s+\delta$ pressure
\[
{\Psi}^k_\omega(s+\delta,\epsilon)=\vertiii{\mathds{1} \mathbf{P}^{s+\delta}_{\epsilon}(\omega) \hdots  \mathbf{P}^{s+\delta}_{\epsilon}(\sigma^{k-1}\omega)}^{1/k}
\]
\[
\geq\underline\gamma^\delta\vertiii{\mathds{1} \mathbf{P}^{s}_{\epsilon}(\omega)  \hdots \mathbf{P}^{s}_{\epsilon}(\sigma^{k-1}\omega) }^{1/k}
\geq\underline\gamma^{\delta}\Psi^{k}_{\omega}(s,\epsilon),
\]
and similarly for the upper bound we have ${\Psi}^k_\omega(s+\delta,\epsilon)\leq\overline\gamma\Psi^{k}_{\omega}(s,\epsilon)$.
Therefore, if the limit exists, $\underline\gamma^{\delta}\mathbf{\Psi}_{\omega}(s,\epsilon)\leq\mathbf{\Psi}_{\omega}(s+\delta,\epsilon)\leq\overline\gamma^{\delta}\mathbf{\Psi}_{\omega}(s,\epsilon)$. Thus as $0<\underline\gamma\leq\overline\gamma<1$, $\mathbf{\Psi}_{\omega}(s,\epsilon)$ is strictly decreasing in $s$ and, taking $\delta\to0$, is easily seen to be continuous for almost every $\omega$ and thus $\mathbf{\Psi}(s,\epsilon)$ has the same property.
Letting $\delta\to\infty$ we see $\mathbf{\Psi}(s+\delta,\epsilon)\to0$ and $\mathbf{\Psi}(0,\epsilon)\geq1$ by the non-extinction given by Condition~\ref{nontrivialgraph}. The existence and uniqueness of $s_{H,\epsilon}$ then follows.
\qed

\subsection{Proof of Theorem \ref{upperBox}}\label{upperBoxProof}
Note that the proof below directly implies that the box dimension exists almost surely.

Our argument relies on a supermultiplicative property of approximations of $\epsilon$-stopping graphs given by~(\ref{submultEq}). Before we derive that expression we establish a connection between the least number of sets of diameter $\epsilon$ or less needed to cover our attractor $N_{\epsilon}(K_{v}(\omega))$ and the number of edges of our $\epsilon$-stopping graph $\lvert \tensor[_{v}]{E}{}(\omega,\epsilon)\rvert$.
By the definition of the $\epsilon$-stopping graph we have that for all $\mathbf{e}\in\tensor[_{v}]{E}{}(\omega,\epsilon)$ the diameter of $S_\mathbf{e}(\Delta)$ is of order $\epsilon$, see Definition~\ref{stoppingGraphDefn}. 
Since we also have that the images of the stopping $\{S_\mathbf{e}(\Delta)\}_{\mathbf{e}\in\tensor[_{v}]{E}{}(\omega,\epsilon)}$ are pairwise disjoint, $\{S_\mathbf{e}(\Delta)\}_{\mathbf{e}\in\tensor[_{v}]{E}{}(\omega,\epsilon)}$ may not form a cover of $K_{v}(\omega)$.
But since the construction is maximal, the image of any word (edge) that was deleted must intersect another image of a word that was kept, which means that to form a cover of $K_{v}(\omega)$ one needs at most $3^{d}\lceil\underline \gamma^{-1}\rceil\lvert \tensor[_{v}]{E}{}(\omega,\epsilon)\rvert$ $d$-dimensional hypercubes of sidelength $\epsilon$ to form a cover and hence $N_{\epsilon}(K_{v}(\omega))\leq 3^{d}\lceil\underline \gamma^{-1}\rceil\lvert \tensor[_{v}]{E}{}(\omega,\epsilon)\rvert$.
On the other hand, any element in the minimal cover for $N_{\epsilon}(K_{v}(\omega))$ can intersect at most a uniformly bounded number of elements in $\{S_\mathbf{e}(\Delta)\}_\mathbf{e}$, as otherwise the elements in $\{S_\mathbf{e}(\Delta)\}_\mathbf{e}$ would intersect. Hence there exists $k_{\min}>0$ such that 
$N_{\epsilon}(K_{v}(\omega))\geq k_{\min}\lvert \tensor[_{v}]{E}{}(\omega,\epsilon)\rvert$ and we get the required 
\begin{equation}\label{boxasymp}
N_{\epsilon}(K_{v}(\omega))\asymp\lvert \tensor[_{v}]{E}{}(\omega,\epsilon)\rvert.
\end{equation}
Using the notation of the Hutchinson-Moran matrices introduced in (\ref{HutchinsonMatrix}), we can see that for $s=0$, we have $c_\mathbf{e}^{s}=1$ and thus the Hutchinson matrix $\mathbf{P}_{\epsilon}^{0}(\omega)$ `counts' the number of images in $\tensor[]{E}{}(\omega,\epsilon)$. We have
\[
 \lvert \tensor[_{v}]{E}{}(\omega,\epsilon)\rvert=\sum_{w\in V}\left(\sum_{j}(\mathds{1}\mathbf{P}_{\epsilon}^{0}(\omega))_{j}\right)_{(v,w)}.
\]
The sum above behaves in a supermultiplicative fashion: for some constant $k_{s}>0$ and all $\epsilon,\delta>0$,
\begin{equation}\label{submultEq}
\sum_{w\in V}\left(\sum_{j}(\mathds{1}\mathbf{P}_{\delta\epsilon}^{0}(\omega))_{j}\right)_{(v,w)}\geq k_{s}\sum_{w\in V}\left(\sum_{j}(\mathds{1}\mathbf{P}_{\epsilon}^{0}(\omega)\mathbf{P}_{\delta}^{0}(\sigma\omega))_{j}\right)_{(v,w)}.
\end{equation}
By definition $\bigoplus\mathds{1}_{\epsilon_{0}}\mathbf{H}^{\epsilon}(\omega)$ is the arrangement of words that describe the cylinders of $\{K_{v}(\omega)\}_{v\in V}$. Consider an arbitrary word $\mathbf{e}_{1}\mathbf{e}_{2}\in\bigoplus\mathds{1}_{\epsilon_{0}}\mathbf{H}^{\epsilon}(\omega)\mathbf{H}^{\delta}(\sigma\omega)$, where $\mathbf{e}_{1}\in\bigoplus\mathds{1}_{\epsilon_{0}}\mathbf{H}^{\epsilon}(\omega)$ and $\mathbf{e}_{2}\in\bigoplus\mathbf{H}^{\delta}(\sigma\omega)$. 
Assume $\mathbf{e}_{1}$ is the $(i,j)$th entry of the matrix at position $k$ of the vector $\mathds{1}_{\epsilon_{0}}\mathbf{H}^{\epsilon}(\omega)$. 
Since $\mathbf{e}_{1}\mathbf{e}_{2}$ is obtained by regular matrix multiplication, we have that $\mathbf{e}_{2}$ is an entry in one of the matrices in the $k$-th row of $\mathbf{H}^{\delta}(\sigma\omega)$, $\mathbf{e}_{2}\in \hat\eta^{k}(\sigma^{k}\omega,\delta)$. Therefore, for some $v_{1},v_{2},v_{3}\in V$, we have $\mathbf{e}_{1}\in\tensor*[_{v_{1}}]{E}{^{k}_{v_{2}}}(\omega,\epsilon)$ and $\mathbf{e}_{2}\in\tensor*[_{v_{2}}]{E}{_{v_{3}}}(\sigma^{k}\omega,\delta)$. Hence $\mathbf{e}_{1}\mathbf{e}_{2}$ describes a path of $\mathbf{\Gamma}$ for realisation $\omega$ and therefore codes a cylinder of $K_{v_{1}}(\omega)$, and as $c_{\min}\epsilon<c_{\mathbf{e}_{1}}\leq\epsilon$ and $c_{\min}\delta<c_{\mathbf{e}_{2}}\leq\delta$ we additionally have $c_{\min}^{2}\delta\epsilon<c_{\mathbf{e}_{1}\mathbf{e}_{2}}\leq\delta\epsilon$. Recall that $\mathfrak{R}^{s}$ was the operator mapping arrangements of words to the length of the associated image under $S$ to the power $s$. Therefore, applying $\mathfrak{R}^{0}$ to $(\mathds{1}_{\epsilon_{0}}\mathbf{H}^{\epsilon}(\omega)\mathbf{H}^{\delta}(\sigma\omega))$, we can express the number of cylinders starting at a given vertex $v$ by
\[
\sum_{w\in V}\left(\sum_{j}(\mathds{1}\mathbf{P}_{\epsilon}^{0}(\omega)\mathbf{P}_{\delta}^{0}(\sigma\omega))_{j}\right)_{(v,w)}.
\]
Obviously these cylinders do not intersect but they do not quite form an $\epsilon\delta$-stopping graph as some of the edges might have contraction rate 
$c_{\min}^{2}\delta\epsilon<c_{\mathbf{e}_{1}\mathbf{e}_{2}}\leq c_{\min}\delta\epsilon$. However this does not present a problem as one needs to only avoid at most the last branching to recover an $\epsilon\delta$-stopping graph. Note that the number of subbranches is surely bounded and therefore there exists a constant $k_{s}$, which is the inverse of this maximal splitting bound, such that we have an $\epsilon\delta$-stopping graph that may not be maximal, hence giving rise to the inequality~(\ref{submultEq}).

\vspace{0.5cm}
Now given any $\epsilon>\delta>0$ there exists unique $q\in\N$ and $1\geq\xi>\epsilon$ such that $\delta=\epsilon^{q}\xi$. One can easily generalise equation~(\ref{submultEq}), using above argument, to show that
\begin{equation}\label{generalSubMult}
\sum_{w\in V}\left(\sum_{j}(\mathds{1}\mathbf{P}_{\delta}^{0}(\omega))_{j}\right)_{(v,w)}\geq k_{s}^{q}\sum_{w\in V}\left(\sum_{j}(\mathds{1}\mathbf{P}_{\xi}^{0}(\omega)\mathbf{P}_{\epsilon}^{0}(\sigma\omega)\dots \mathbf{P}_{\epsilon}^{0}(\sigma^{q-1}\omega))_{j}\right)_{(v,w)}.
\end{equation}
The relationship between the expression above and the exponent $\epsilon$ can be found by an argument akin to that in the proof of Fekete's Lemma, see \cite[\S1 Problem 98]{PolyAnalysisProblems}. Consider
\begin{align}
\frac{\log\sum_{w\in V}\left(\sum_{j}(\mathds{1}\mathbf{P}_{\delta}^{0}(\omega))_{j}\right)_{(v,w)}}{-\log\delta}
&=\frac{\log\sum_{w\in V}\left(\sum_{j}(\mathds{1}\mathbf{P}_{\epsilon^{q}\xi}^{0}(\omega))_{j}\right)_{(v,w)}}{-q\log\epsilon-\log\xi}\nonumber
\end{align}
\begin{align}
&\geq \frac{\log k_{s}+\log\left(\sum_{w\in V}\left(\sum_{j}(\mathds{1}\mathbf{P}_{\xi}^{0}(\omega)\mathbf{P}_{\epsilon}^{0}(\sigma\omega)\dots \mathbf{P}_{\epsilon}^{0}(\sigma^{q}\omega))_{j}\right)_{(v,w)}\right)^{1/q}}{-\log\epsilon-(1/q)\log\xi}\nonumber.
\end{align}
Thus for every $\epsilon>0$, assuming almost sure convergence and stochastically strongly connected graphs,  
\[
\liminf_{\delta\to0}\frac{\log\sum_{w\in V}\left(\sum\mathds{1}\mathbf{P}_{\delta}^{0}(\omega)\right)_{(v,w)}}{-\log\delta}\geq
\lim_{\epsilon\to0}\frac{\log k_{s}+\log\mathbf{\Psi}(0,\epsilon)}{-\log\epsilon}
\geq \sup_{\epsilon>0}\frac{\log\mathbf{\Psi}(0,\epsilon)}{-\log\epsilon},
\]
holding almost surely. For the upper bound simply note that, almost surely,
\[
\limsup_{\delta\to0}\frac{\log\sum_{w\in V}\left(\sum\mathds{1}\mathbf{P}_{\delta}^{0}(\omega)\right)_{(v,w)}}{-\log\delta}\leq\sup_{\delta>0}\frac{\log\mathbf{\Psi}(0,\delta)}{-\log\delta}.
\]
Therefore, almost surely,
\[
\frac{\log\sum_{w\in V}\left(\sum\mathds{1}\mathbf{P}_{\delta}^{0}(\omega)\right)_{(v,w)}}{-\log\delta}\to \sup_{\epsilon>0}\frac{\log\mathbf{\Psi}(0,\epsilon)}{-\log\epsilon}\;\;\;\text{ as }\delta\to\infty.
\]
Due to~(\ref{boxasymp}) we get the required almost sure result:
\[
\dim_{B}K_{v}(\omega)=\sup_{\epsilon>0}\frac{\log\mathbf{\Psi}(0,\epsilon)}{-\log\epsilon}.
\]
\qed

\subsection{Proof of Theorem \ref{lowerHausdorff}}
While the construction introduced in Section~\ref{infiniteMatrixSection} with norm $\vertiii{.}$ makes sense in establishing the box counting dimension of RGDS attractors where we wanted all cylinders of diameter comparable to some $\epsilon>0$, we can also rewrite the system as a finite graph directed system. 
We employ this idea here to find the lower bound to the Hausdorff dimension of $K_{v,\epsilon}(\omega)$ by constructing a measure on cylinders obtained in this finite fashion. 
Since $K_{v,\epsilon}(\omega)\subseteq K_{v}(\omega)$, the Hausdorff dimension for the approximation will give a lower bound for the Hausdorff dimension of $K_{v}(\omega)$. We shall use the $\vertiii{.}_{(1,1)}$ seminorm defined in (\ref{square1norm}) on finite matrices with matrix entries.

Consider the system given by the states $A_{1}, H_{2}, H_{3}, \hdots, H_{k_{\max}(\epsilon)}$, where $k_{\max}(\epsilon)$ is the maximal length of column specified by $\epsilon$, see Section~\ref{infiniteMatrixSection}. The corresponding graph is shown in Figure~\ref{waitingTimeFigure}. We record words in either the active ($A_{1}$) or a holding state ($H_{i}$) as a $k_{\max}(\epsilon)$-vector with matrix entries and the action given from the active state by right multiplication of $\mathbf{C}_{\epsilon}(\omega)$ and $\mathbf{W}_{\epsilon}^{s}(\omega)=\mathfrak{R}^{s}\mathbf{C}_{\epsilon}(\omega)$, where 
\[
\mathbf{C}_{\epsilon}(\omega)=\begin{pmatrix}
\eta_{1}(\omega,\epsilon)	& \eta_{2}(\omega,\epsilon)	&\hdots	&	\eta_{k_{\max}(\omega)}(\omega,\epsilon)\\
\mathbf{1}_{\epsilon_{0}}	& \mathbf{0}_{\varnothing}	&\hdots	&	\mathbf{0}_{\varnothing}\vphantom{\vdots}\\
\mathbf{0}_{\varnothing}	& \mathbf{1}_{\epsilon_{0}}	&\hdots	&	\mathbf{0}_{\varnothing}\vphantom{\vdots}\\
\vdots	&	&	\ddots&\vdots\\
\mathbf{0}_{\varnothing}	&\mathbf{0}_{\varnothing}	&\hdots&\mathbf{1}_{\epsilon_{0}}\vphantom{\vdots}\\
\mathbf{0}_{\varnothing}	&\mathbf{0}_{\varnothing}	&\hdots&\mathbf{0}_{\varnothing}\vphantom{\vdots}
\end{pmatrix}
\]
and
\[
\mathbf{W}^{s}_{\epsilon}(\omega)=\begin{pmatrix}
\mathbf{p}^{s}_{1}(\omega,\epsilon)	 &\hdots	&	\mathbf{p}^{s}_{k_{\max}(\omega)}(\omega,\epsilon)\\
\mathbf{1}	&\hdots	&	\mathbf{0}\vphantom{\vdots}\\
\vdots	&	\ddots&\vdots\\
\mathbf{0}	&\hdots&\mathbf{1}\vphantom{\vdots}\\
\mathbf{0}	&\hdots&\mathbf{0}\vphantom{\vdots}
\end{pmatrix}.
\]
We are now interested in analysing the cylinders given by the (finite) arrangement of words $\mathbf{D}_{\epsilon}^{k}(\omega)$ and the norm of its Hutchinson-Moran matrix $\mathfrak{R}^{s}\mathbf{D}_{\epsilon}^{k}(\omega)$,
\[
\mathbf{D}_{\epsilon}^{k}(\omega)=\mathds{1}_{\epsilon_{0}}\mathbf{C}_{\epsilon}(\omega)\mathbf{C}_{\epsilon}(\sigma\omega)\dots\mathbf{C}_{\epsilon}(\sigma^{k-1}\omega) \text{ and }
{\Phi}^{k}_{\epsilon}(s)=\vertiii{\mathds{1}\mathbf{W}^{s}_{\epsilon}(\omega)\dots\mathbf{W}^{s}_{\epsilon}(\sigma^{k-1}\omega)}_{(1,1)}.
\]
\begin{figure}
\label{waitingTimeFigure}
\begin{tikzpicture}[->,>=stealth',shorten >=1pt,auto,
  thick,main node/.style={circle,fill=blue!20,draw,font=\sffamily\Large\bfseries}]

\node(1){$A_{1}$};
\node(2)[right of=1,node distance=4cm]{$H_{2}$};
\node(3)[above of=2,node distance=1.5cm]{$H_{3}$};
\node(4)[above of=3,node distance=1.5cm]{$H_{4}$};
\node(5)[above of=4,node distance=1.5cm]{$\vdots$};
\node(6)[above of=5,node distance=1.5cm]{$H_{l}$};

\path[every node/.style={font=\sffamily\small}]
(1)edge node [near end]{${\eta_{2}}(\omega)$}(2)
edge node [very near end] {$\eta_{3}(\omega)$}(3)
edge node [near end]{${\eta_{4}}(\omega)$}(4)
edge node [midway] {${\eta_{l}}(\omega)$}(6)
edge [loop left] node {${\eta_{1}}(\omega)$}(1)
(2) edge [bend left] node {$\mathbf{1}_{\epsilon_{0}}$}(1)
(3) edge [bend left] node {$\mathbf{1}_{\epsilon_{0}}$}(2)
(4) edge [bend left] node {$\mathbf{1}_{\epsilon_{0}}$}(3)
(5) edge [bend left] node {$\mathbf{1}_{\epsilon_{0}}$}(4)
(6) edge [bend left] node {$\mathbf{1}_{\epsilon_{0}}$}(5)
;
\end{tikzpicture}\hspace{0.0cm}
\begin{tikzpicture}[->,>=stealth',shorten >=1pt,auto,
  thick,main node/.style={circle,fill=blue!20,draw,font=\sffamily\Large\bfseries}]

\node(1){$A_{1}$};
\node(2)[right of=1,node distance=4cm]{$H_{2}$};
\node(3)[above of=2,node distance=1.5cm]{$H_{3}$};
\node(4)[above of=3,node distance=1.5cm]{$H_{4}$};
\node(5)[above of=4,node distance=1.5cm]{$\vdots$};
\node(6)[above of=5,node distance=1.5cm]{$H_{l}$};

\path[every node/.style={font=\sffamily\small}]
(1)edge node [near end]{${\mathbf{p}^{s}_{2}}(\omega)$}(2)
edge node [very near end] {$\mathbf{p}^{s}_{3}(\omega)$}(3)
edge node [near end]{${\mathbf{p}^{s}_{4}}(\omega)$}(4)
edge node [midway] {${\mathbf{p}^{s}_{l}}(\omega)$}(6)
edge [loop left] node {${\mathbf{p}^{s}_{1}}(\omega)$}(1)
(2) edge [bend left] node {$\mathbf{1}$}(1)
(3) edge [bend left] node {$\mathbf{1}$}(2)
(4) edge [bend left] node {$\mathbf{1}$}(3)
(5) edge [bend left] node {$\mathbf{1}$}(4)
(6) edge [bend left] node {$\mathbf{1}$}(5)
;
\end{tikzpicture}
\caption{Graph for the finite model used in establishing the lower bound.}
\end{figure}
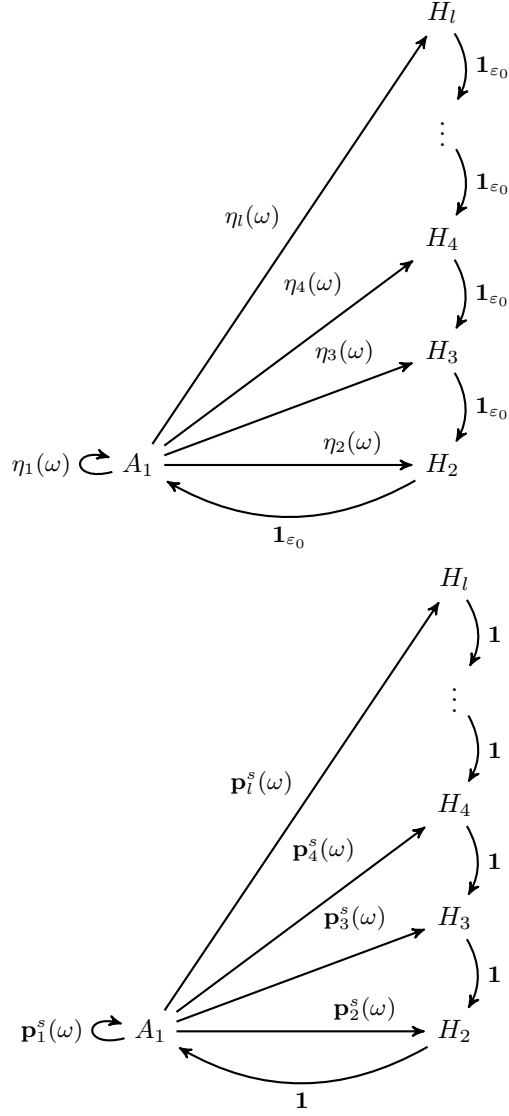
We first show
\begin{lma}
On a subset of $\Omega$ with full measure we have, for all $\epsilon>0$,
\[
\mathbf{\Phi}_{\epsilon}(s):=\lim_{k\to\infty}({\Phi}^{k}_{\epsilon}(s))^{1/k}=1\;\;\;\text{ if and only if }\;\;\; \mathbf{\Psi}(s,\epsilon)=1.
\]
\end{lma}
Note that these two notions of pressure do not, in general, coincide for $s$ when $\mathbf{\Phi}_{\epsilon}(s)\neq1$.
\begin{proof}
The procedure of picking the multiplications that are applied to the active state $A_{1}$ is determined by the first $k_{\max}(\epsilon)$ letters of $\omega$, where the individual entries of $\omega$ were chosen independently from $\Lambda$ according to $\vec \pi$. However, one can without loss of generality assume that the matrices picked are given by a stochastic process that is Markov. To see this let $\Lambda^{\ddagger}$ be a new alphabet consisting of $\lvert\Lambda\rvert^{k_{\max}(\epsilon)}$ elements. These elements represent all the different strings one can have that determine the matrices chosen. The full shift on $\Omega$ now induces a subshift of finite type on $(\Lambda^{\ddagger})^{\N}$ and $\mu$ gives a new Markov measure $\mu^{\ddagger}$ with appropriate transition probabilities. It is a simple exercise to show that this subshift is also topologically mixing and we omit it here.

The cylinders given by $\mathbf{D}_{\epsilon}^{k}(\omega)$ still exhaust all paths (compare with Lemma~\ref{lma:stoppingrepl}), however they may no longer have comparable diameter. Given that it is a stopping set we can find certain inclusions if we compare the arrangement of words of this finite model with the arrangement of words coming from the infinite construction. Let $\mathbf{U}^{k}_{\epsilon}(\epsilon)=\mathbf{H}^{\epsilon}(\omega)\mathbf{H}^{\epsilon}(\sigma\omega)\dots\mathbf{H}^{\epsilon}(\sigma^{k-1}\omega)$. Then
\begin{equation}\label{complicatedInclusion}
\bigoplus_{i=1}^{\lfloor (k+1)/l\rfloor+1}(\mathds{1}_{\epsilon_{0}}\mathbf{U}_{\epsilon}^{k-i-1}(\omega))_{i}
\subseteq
\bigoplus \mathbf{D}_{\epsilon}^{k}(\omega).
\end{equation}
To see this inclusion we refer the reader back to Figure~\ref{layeredWords}. The arrangement $\mathbf{D}_{\epsilon}^{k}(\omega)$ corresponds to taking the off-diagonal of entries that have been decided up to the $k$th shift.
The left hand side of (\ref{complicatedInclusion}) are exactly those words that were in state $A_{1}$ at the $(k-1)$th shift and are part of the same off-diagonals in Figure~\ref{layeredWords}.

The diagonal must also intersect with an element that is within some uniform constant $c>0$ of the maximal element on some level $d_{k}$ from $\lfloor (k+1)/l\rfloor+1$ to $k$, giving the following inclusion:
\[
\bigoplus \mathbf{D}_{\epsilon}^{k}(\omega)
\subseteq
\bigoplus_{i=\lfloor (k+1)/l\rfloor+1}^{k} \;\;\bigoplus_{j\in\N} (\mathds{1}_{\epsilon_{0}}\mathbf{U}_{\epsilon}^{i}(\omega))_{j}.
\]
Applying the operator $\mathfrak{R}^{s}$ we get the inequalities
\begin{align}
\sum_{i=1}^{\lfloor (k+1)/l\rfloor+1}\lVert(\mathds{1}\mathbf{u}_{k-i-1}(\omega))_{i}\rVert_{1}
\leq
\Phi^{k}_{\epsilon}(s)
&\leq
\sum_{i=\lfloor (k+1)/l\rfloor+1}^{k}\;\; \sum_{j\in\N} \lVert(\mathds{1}\mathbf{u}_{i}(\omega))_{j}\rVert_{1}\label{eqn:diagonalbounds}\\
&\leq
n\sum_{i=\lfloor (k+1)/l\rfloor+1}^{k}\vertiii{\mathds{1}\mathbf{u}_{i}(\omega)}
.\nonumber
\end{align}
Let $m_{k}$ refer to the level for which $\mathbf{\Psi}_{m_{k}}(s)=\max_{i\in\{\lfloor (k+1)/l\rfloor+1,\dots,k\}}\mathbf{\Psi}_{i}(s)$ and $d_{k}$ be as above, then~(\ref{eqn:diagonalbounds}) becomes
\begin{align}
\lVert(\mathds{1}\mathbf{u}_{d_{k}}(\omega))_{d_{k}}\rVert_{1}&\leq\Phi_{\epsilon}^{k}(s)\leq nk\vertiii{\mathds{1}\mathbf{u}_{m_{k}}(\omega)}\nonumber\\
\frac{1}{k}\Psi_{\omega}^{d_{k}}(s,\epsilon)&\leq\Phi_{\epsilon}^{k}(s)\leq nk\Psi^{m_{k}}_{\omega}(s,\epsilon)\nonumber\\
k^{-1/k}\Psi_{\omega}^{d_{k}}(s,\epsilon)^{1/k}&\leq\Phi_{\epsilon}^{k}(s)^{1/k}\leq (nk)^{1/k}\Psi^{m_{k}}_{\omega}(s,\epsilon)^{1/k}.\nonumber
\end{align}
Now assume $s$ is such that $\mathbf{\Psi}_{\omega}(s,\epsilon)=1$ for all $\omega\in \mathcal U$, where $\mathcal U$ is a set of measure one. Now,
\[
\limsup_{k}{\Phi}^{k}_{\epsilon}(s)^{1/k}\leq \limsup_{k}(nk)^{1/k}{\Psi}_{\omega}^{m_{k}}(s,\epsilon)^{1/k}\leq\limsup_{k} k^{1/k}{\Psi}_{\omega}^{m_{k}}(s,\epsilon)^{1/m_{k}}=1
\]
And similarly
\[
\liminf_{k}{\Phi}_{\epsilon}^{k}(s)^{1/k}\geq\liminf_{k} k^{-1/k}{\Psi}_{\omega}^{d_{k}}(s,\epsilon)^{1/k}\geq\liminf_{k} k^{-1/k}{\Psi}_{\omega}^{d_{k}}(s,\epsilon)^{1/d_{k}}=1
\]
Thus $\mathbf{\Psi}(s,\epsilon)=1\Rightarrow\mathbf{\Phi}_{\epsilon}(s)=1$. To establish the other direction just note that if $s$ is such that $\mathbf{\Psi}(s,\epsilon)<1$, then eventually ${\Psi}_{\omega}^{k}(s,\epsilon)^{1/k}\leq1-\delta$ for all $\omega\in\mathcal{U}$ and $\delta>0$ and $k$ large enough and so ${\Psi}_{\omega}^{k'}(s,\epsilon)\leq 1-\delta$ for large enough $k'\geq k$. This gives 
\begin{align}
\limsup_{k}{\Phi}_{k}(s)^{1/k} &\leq\limsup_{k} k^{1/k}\Psi_{\omega}^{m_{k}}(s,\epsilon)^{1/k}\nonumber\\
&\leq\limsup_{k} k^{1/k}\Psi_{\omega}^{m_{k}}(s,\epsilon)^{1/(lm_{k}+l+1)}\nonumber\\
&\leq(\limsup_{k} k^{1/k}\Psi_{\omega}^{m_{k}}(s,\epsilon)^{1/m_{k}})^{1/(l+1)}<1\nonumber.
\end{align}
A similar argument holds for $\mathbf{\Psi}(s,\epsilon)>1$, finishing the proof.
\end{proof}

For $t<s_{H,\epsilon}$ we can define a random mass distribution on $K_{v,\epsilon}(\omega)$ by constructing a Borel probability measure $\nu$ on the cylinders described by $\mathbf{D}_{\epsilon}^{k}(\omega)$ that satisfies $\nu(U)\leq C\lvert U\rvert^{t}$ for some random, almost surely non-zero, constant $C$.
We start by defining the (diagonal) $k$-prefractal codings of $K_{v,\epsilon}(\omega)$ for the vertex $v$ by 
\[
\mathcal F_{k}^{v}(\omega)=\bigoplus_{w\in V}\left(\bigoplus_{j=1}^{l}(\mathds{1}_{\epsilon_{0}}\mathbf{D}_{\epsilon}^{k}(\omega))_{i}\right)_{v,w}
\]
Since the words of $\mathcal F_{k}^{v}(\omega)$ are in one to one correspondence with the cylinders generating the topology on $K_{v}(\omega)$ it suffices to define our required measure on those (disjoint) cylinders only, as they generate the topology of $K_{v}(\omega)$ and this construction extends to a unique Borel probability measure  $\nu^{s}_{v}$.
For every word $w\in\mathcal F_{k}^{v}(\omega)$ we can describe its `location' relative to $\mathbf{D}_{\epsilon}^{k}(\omega)$ by a unique triple $(x,y,z)$, where $x,y\in V$ and $z\in\{1,\dots,l\}$, such that $w\in [(\mathds{1}_{\epsilon_{0}}\mathbf{D}_{\epsilon}^{k}(\omega))_{z}]_{x,y}$. Let $I$ be an arbitrary word in $\mathcal F_{k}^{v}(\omega)$, with coordinates $(x,y,z)$. 
For any word we define the \emph{location matrix} as 
\[
(\mathds{V}(I))_{i}=\begin{cases}
V(I)	&	\text{ for }i=z,\\
\mathbf{0}_{\varnothing}	&	\text{ otherwise};
\end{cases}
\;\;\;\text{for}\;\;\;
(V(I))_{j,k}=\begin{cases}
I	&	\text{ for }(j,k)=(x,y)\\
\varnothing	&	\text{ otherwise.}
\end{cases}
\]
We set for $I\in\mathcal{F}_{k}^{v}(\omega)$,
\begin{equation}\label{measureDef}
\nu_{v}^{s}(I)=
\lim_{q\to\infty}
\frac{\vertiii{\mathfrak{R}^{s}\left(\mathds{V}(I)\mathbf{C}_{\epsilon}(\sigma^{k}\omega)\mathbf{C}_{\epsilon}(\sigma^{k+1}\omega)\dots\mathbf{C}_{\epsilon}(\sigma^{k+q-1}\omega)\right)}_{(1,1)}}
{\sum_{q_{2}=1}^{n}\left[\sum_{q_{1}=1}^{l}(\mathds{1}_{l}\mathbf{W}_{\epsilon}^{s}(\omega)\mathbf{W}_{\epsilon}^{s}(\sigma\omega)\dots
\mathbf{W}_{\epsilon}^{s}(\sigma^{q-1}\omega) )_{q_{1}}\right]_{v,q_{2}}}.
\end{equation}
One can check that, almost surely, this limit exists. However as one can derive the properties of $\nu_{v}^{s}$ by defining the measure in terms of $\liminf$ or $\limsup$, we omit details.
It is easy to see that $\nu_{v}^{s}$ is in fact a measure. Note that for $I=\varnothing$ we get $\mathfrak{R}^{s}\mathds{V}(I)=\mathbf{0}$ and so $\nu_{v}^{s}(\varnothing)=0$. Obviously $\nu_{v}^{s}(I)\geq0$
and countable stability arises from the construction being an additive set function, where
\[
\nu_{v}^{s}(I)=\lim_{k\to\infty}\left\{\sum \lvert J\rvert^{s} \mid J\in\mathcal{F}_{k}^{v}(\omega)\text{ and }J\subseteq I\right\}.
\]
Formally, for any countable collection of disjoint words (no word is a subword of any other) $\bigoplus w_{i}$ we get, assuming that $w_{i}\in \mathcal{F}_{k_{i}}^{v}(\omega)$ for some length $k_{i}$,
\begin{align}
\sum_{i}\nu_{v}^{s}([w_{i}])&=\sum_{i}\lim_{q\to\infty}\frac{\vertiii{\mathfrak{R}^{s}\left(\mathds{V}(w_{i})\mathbf{C}_{\epsilon}(\sigma^{k_{i}}\omega)\dots\mathbf{C}_{\epsilon}(\sigma^{k_{i}+q-1}\omega)\right)}_{(1,1)}}
{\sum_{q_{2}=1}^{n}\left[\sum_{q_{1}=1}^{l}(\mathds{1}_{l}\mathbf{W}_{\epsilon}^{s}(\omega)\dots
\mathbf{W}_{\epsilon}^{s}(\sigma^{q-1}\omega) )_{q_{1}}\right]_{v,q_{2}}}\nonumber\\
&=\lim_{q\to\infty}\frac{\vertiii{\mathfrak{R}^{s}\left(\bigoplus_{i}\mathds{V}(w_{i})\mathbf{C}_{\epsilon}(\sigma^{k_{i}}\omega)\dots\mathbf{C}_{\epsilon}(\sigma^{k_{i}+q-1}\omega)\right)}_{(1,1)}}
{\sum_{q_{2}=1}^{n}\left[\sum_{q_{1}=1}^{l}(\mathds{1}_{l}\mathbf{W}_{\epsilon}^{s}(\omega)\dots
\mathbf{W}_{\epsilon}^{s}(\sigma^{q-1}\omega) )_{q_{1}}\right]_{v,q_{2}}}\nonumber\\
&=\nu_{v}^{s}\left(\left[\bigoplus_{i}w_{i}\right]\right).\nonumber
\end{align}

Notice that there exists a uniform constant $C>0$ such that
\[
\nu_{v}^{s}(K_{v,\epsilon}(\omega))=\lim_{q\to\infty}
\frac{\vertiii{\mathfrak{R}^{s}\left(\mathds{1}_{l}\mathbf{C}_{\epsilon}(\omega)\mathbf{C}_{\epsilon}(\sigma\omega)\dots\mathbf{C}_{\epsilon}(\sigma^{q-1}\omega)\right)}_{(1,1)}}
{\sum_{q_{2}=1}^{n}\left[\sum_{q_{1}=1}^{l}(\mathds{1}_{l}\mathbf{W}_{\epsilon}^{s}(\omega)\mathbf{W}_{\epsilon}^{s}(\sigma\omega)\dots
\mathbf{W}_{\epsilon}^{s}(\sigma^{q-1}\omega) )_{q_{1}}\right]_{v,q_{2}}}\leq C
\]
and we conclude that $\nu_{v}^{s}$ is a finite measure, and without loss of generality we rescale such that $\nu_{v}^{s}=1$.

We observe that by virtue of the definition of the measure that there exists a random variable $C^{\dagger}(\omega)$ with $\E_{\omega}C^{\dagger}(\omega)<\infty$ such that 
\begin{equation}\label{measureScaling}
\nu_{v}^{s}(I)\leq C C^{\dagger}(\omega)\lvert I\rvert^{s}
\end{equation}
as long as $s<s_{H,\epsilon}$, such that the denominator in~(\ref{measureDef}) is almost surely increasing exponentially in $q$. Note that the existence of a Borel measure satisfying~(\ref{measureScaling}) immediately implies that $s_{H,\epsilon}$ is an almost sure lower bound by the mass distribution principle.\qed

\subsection{Proof of Theorem \ref{dimensionEquality} and Corollary \ref{separatedCorollary}}

\subsubsection{Proof of Theorem \ref{dimensionEquality}}
Let $\lambda>0$, Theorem~\ref{lowerHausdorff} gives us a lower bound on the Hausdorff dimension of the $\lambda$-approximation sets $K_{v,\lambda}(\omega)$. In particular we have that $\dim_{H}K_{v,\lambda}=s_{H,\lambda}$, where
\[
\lim_{k\to\infty}\vertiii{\mathds{1}\mathbf{P}^{s_{H,\lambda}}_{\lambda}(\omega)\dots \mathbf{P}^{s_{H,\lambda}}_{\lambda}(\sigma^{k-1}\omega)}^{1/k}=1.
\]
Consider one of the Hutchinson-Moran sums in the matrix $\mathbf{P}^{s_{H,\lambda}}_{\lambda}(\omega)$. They are given by 
\[
\sum_{e\in(\tensor*[_{i}]{E}{^{q}_{j}}(\omega,\lambda))}c_{\mathbf{e}}^{s_{H,\lambda}}.
\]
But since we have bounds on the size of $c_\mathbf{e}$, i.e.\ $\underline\gamma\lambda<c_\mathbf{e}\leq\lambda$ we have 
\[
\sum_{e\in(\tensor*[_{i}]{E}{^{q}_{j}}(\omega,\lambda))}c_{e}^{s_{H,\lambda}}\leq \lvert\tensor*[_{i}]{E}{^{q}_{j}}(\omega,\lambda)\rvert \lambda^{s_{H,\lambda}}\]
 and so
\[
\mathbf{P}^{s_{H,\lambda}}_{\lambda}(\omega)\leq \lambda^{s_{H,\lambda}}\mathbf{P}^{0}_{\lambda}(\omega).
\]
Considering the matrices $\lambda^{s}\mathbf{P}^{0}_{\lambda}(\omega)$, dependent on $s$, one can apply the same strategy as in Lemma~\ref{almostSurePressureProperties} to prove that there exists a unique $0\leq t_{\lambda}\leq s_{H,\lambda}$ such that 
\[
\lim_{k\to\infty}\vertiii{\mathds{1}\lambda^{t_{\lambda}}\mathbf{P}^{0}_{\lambda}(\omega)\lambda^{t_{\lambda}}\mathbf{P}^{0}_{\lambda}(\sigma\omega)\dots \lambda^{t_{\lambda}}\mathbf{P}^{0}_{\lambda}(\sigma^{k-1}\omega)}^{1/k}=1
\]
We leave adapting the proof of Lemma~\ref{almostSurePressureProperties} to the reader. 
Note that the $t_{\lambda}$ defined above gives an a.s.\ lower bound to $\dim_{H}K_{v,\lambda}(\omega)$.
By linearity, 
\begin{align}
\vertiii{\mathds{1}\lambda^{t_{\lambda}}\mathbf{P}^{0}_{\lambda}(\omega)\lambda^{t_{\lambda}}\mathbf{P}^{0}_{\lambda}(\sigma\omega)\dots \lambda^{t_{\lambda}}\mathbf{P}^{0}_{\lambda}(\sigma^{k-1}\omega)}^{1/k}
&=
\vertiii{\mathds{1}\lambda^{kt_{\lambda}}\mathbf{P}^{0}_{\lambda}(\omega) \dots \mathbf{P}^{0}_{\lambda}(\sigma^{k-1}\omega)}^{1/k}
\nonumber\\
&=
\lambda^{t_{\lambda}}\vertiii{\mathds{1}\mathbf{P}^{0}_{\lambda}(\omega) \dots \mathbf{P}^{0}_{\lambda}(\sigma^{k-1}\omega)}^{1/k}		\nonumber
\end{align}
and so 
\[
t_{\lambda}=\lim_{k\to\infty}\frac{\log\vertiii{\mathds{1}\mathbf{P}^{0}_{\lambda}(\omega) \dots \mathbf{P}^{0}_{\lambda}(\sigma^{k-1}\omega)}^{1/k}}{-\log\lambda}
\]
But since $\lim_{k\to\infty}\vertiii{\mathds{1}\mathbf{P}^{0}_{\lambda}(\omega) \dots \mathbf{P}^{0}_{\lambda}(\sigma^{k-1}\omega)}^{1/k}=\mathbf{\Psi}(0,\lambda)$  we have, comparing with~(\ref{upperBoxEq}), that
 \[
t_{\lambda}=\frac{\log\mathbf{\Psi}(0,\lambda)}{-\log\lambda}
 \]
But $t_{\lambda}\to \dim_{B}K_{v}(\omega)$ as $\lambda\to0$ and so we can, for every $\delta>0$, find a $\lambda$ approximation such that, almost surely,  
\[
\overline\dim_{B}K_{v}-\delta\leq\dim_{H}K_{v,\epsilon}\leq\dim_{H}K_{v}\leq\overline\dim_{B}K_{v}.
\]
Therefore $\dim_{H}K_{v}=\overline \dim_{B}K_{v}$ follows for almost all $\omega\in \Omega$.
\qed

\subsubsection{Proof of Corollary \ref{separatedCorollary}}
If our original graph satisfies the USSC, we can apply Lemma~\ref{approxAreSubsets} and have that $K_{v,\epsilon}(\omega)=K_{v}(\omega)$ for all $\epsilon>0$ and $\omega\in\Omega$.
Therefore $s_{H,1}=s_{H}$ and the almost sure Hausdorff, packing and box counting dimensions are given by the unique $s_{O}$ such that
\[
\lim_{k\to\infty}\vertiii{\mathds{1}\mathbf{P}_{1}^{s_{O}}(\omega)\mathbf{P}_{1}^{s_{O}}(\sigma\omega)\dots\mathbf{P}_{1}^{s_{O}}(\sigma^{k-1}\omega)}^{1/k}=1.
\]
But as $\epsilon$ was chosen to be $1$ we must necessarily have $k_{\max}(1)=1$ and $\mathbf{P}_{1}^{s_{O}}(\omega)$ reduces to 
\[
\mathbf{P}_{1}^{s_{O}}(\omega)=
\begin{pmatrix}
p_{1}^{s_{O}}(\omega,1)	&	\mathbf{0}	& \mathbf{0} &\dots\\
\mathbf{0}				&	p_{1}^{s_{O}}(\sigma\omega,1)	& \mathbf{0} &\dots\\
\mathbf{0}				&	\mathbf{0}	&p_{1}^{s_{O}}(\sigma^{2}\omega,1)&\\
\vdots				&\vdots	&	&\ddots
\end{pmatrix}.
\]
But then 
\[
\vertiii{\mathds{1}\mathbf{P}_{1}^{s_{O}}(\omega)\mathbf{P}_{1}^{s_{O}}(\sigma\omega)\dots\mathbf{P}_{1}^{s_{O}}(\sigma^{k-1}\omega)}
=\lVert p_{1}^{s_{O}}(\omega,1)p_{1}^{s_{O}}(\sigma\omega,1)\dots p_{1}^{s_{O}}(\sigma^{k-1}\omega,1)\rVert_{\text{row}}
\]
giving the required result upon noting that $\lVert.\rVert_{\text{row}}$ and $\lVert.\rVert_{1}$ are equivalent norms.
\qed

\subsection{Proof of Theorem~\ref{AssouadTheo}}
The proof of the lower bound is a relatively simple adaptation of the almost sure lower bound proof due to Fraser, Miao and Troscheit~\cite{Fraser14c}. 
%
%

First note that $\mathfrak{P}(\epsilon)$ (see Definition~\ref{jointSpectralDef}) is well-defined by Lemma~\ref{convergenceLemma} since the Lyapunov exponent with respect to the $\vertiii{.}_{\sup}$ norm exists almost surely. 
To see that the joint spectral radius takes the same value recall that 
\[
\mathds{1}\mathbf{P}_{\epsilon}^{0}(\omega)\dots\mathbf{P}_{\epsilon}^{0}(\sigma^{k-1}\omega)
=\left(\mathbf{P}_{\epsilon}^{0}(\omega)\dots\mathbf{P}_{\epsilon}^{0}(\sigma^{k-1}\omega)\right)_{1}
\]
and in general
\[
\mathds{1}\mathbf{P}_{\epsilon}^{0}(\sigma^{l}\omega)\dots\mathbf{P}_{\epsilon}^{0}(\sigma^{k+l-1}\omega)
=\left(\mathbf{P}_{\epsilon}^{0}(\omega)\dots\mathbf{P}_{\epsilon}^{0}(\sigma^{k-1}\omega)\right)_{l}.
\]
However this implies that for almost every $\zeta\in\Omega$ 
\[
\sup_{\omega\in\Omega}\left\{\vertiii{\mathds{1}\mathbf{P}_{\epsilon}^{0}(\omega)\dots\mathbf{P}_{\epsilon}^{0}(\sigma^{k-1}\omega)}\right\}
=\sup_{l\in\N}\left(\mathbf{P}_{\epsilon}^{0}(\zeta)\dots\mathbf{P}_{\epsilon}^{0}(\sigma^{k-1}\zeta)\right)_{l}.
\]
The equality in~(\ref{spectralEquivalence}) thus follows.
Fix $\epsilon>0$ and let $\xi_{i}\in\Omega$ be such that 
\[
\vertiii{\mathds{1}\mathbf{P}_{\epsilon}^{0}(\xi_{i})\dots\mathbf{P}^{0}_{\epsilon}(\sigma^{i-1}\xi_{i})}=\sup_{\omega\in\Omega}\vertiii{\mathds{1}\mathbf{P}_{\epsilon}^{0}(\omega)\dots\mathbf{P}^{0}_{\epsilon}(\sigma^{i-1}\omega)}.
\]
It is easy to check with a standard Borel-Cantelli argument that the set
\[
G=\{\omega\in\Omega \mid \exists \{j_{i}\}_{i=1}^{\infty}\text{ such that }j_{i+1}\geq j_{i}+i, \; \omega_{j_{i}+k_{i}}=\xi_{i}(k_{i}),\text{ for }1\leq k_{i}\leq i\}
\]
has full measure: all finite words $\xi_{i}$ (in increasing order) are subwords of the infinite word $\omega$ with probability $1$. However this is not the actual set that we have to consider. This is because for every $\xi_{i}$ we also associate a row $v_{i}$ as having the maximal sum that is relevant for the norm. Since we however need a result for every row sum to be maximal we have to consider the family of words $\{\xi_{i}^{v}\}$, where $\xi_{i}^{v}=\omega^{v,v_{i}}\xi_{i}$. However this modification does not change the fact that the modified good set
\begin{multline*}
G^{*}=\bigcap_{v\in V}\{\omega\in\Omega \mid \exists \{j_{i}\}_{i=1}^{\infty}\text{ such that }j_{i+1}\geq j_{i}+i+\lvert\omega^{v,v_{i}}\rvert, \\
 \omega_{j_{i}+k_{i}}=\xi_{i}(k_{i}),\text{ for }1\leq k_{i}\leq i+\lvert\omega^{v,v_{i}}\rvert\}
\end{multline*}
still has full measure.

Now assume for a contradiction that 
\[
s:=\dim_{A}(K_{v}(\omega))<t:=\frac{-\log(\mathfrak{P}(\epsilon))}{\log(\epsilon)}.
\]
Let $\{w_{i}\}$ be any sequence of finite words such that the collection of subcylinders $\mathcal{C}(w_{i})$ of $[w_{i}]$ is given by
\[
\mathcal{C}(w_{i})=w_{i}\odot w^{\tau(w_{i}),v_{i}}\odot\left(\bigoplus_{j}\left[\bigoplus_{l}(\mathds{1}_{\epsilon_{0}}\mathbf{H}^{\epsilon}(\xi_{i})\dots \mathbf{H}^{\epsilon}(\sigma^{i-1}\xi_{i}))_{l}\right]_{v_{i},j}\right),
\]
where $w^{a,b}$ is a connecting word from vertex $a$ to $b$, that exists because $\mathbf{\Gamma}$ is stochastically strongly connected.
This sequence of words exists for all $\omega\in G^{*}$, so almost surely,  and we consider the sequence of similitudes given by the (unique) mapping $S^{-1}_{w_{i}}$ that takes the cylinder $[w_{i}]$ and maps it onto $\Delta$.
Consider furthermore the sequence of sets $Z_{i}=S^{-1}_{w_{i}}(K_{v}(\omega))\cap\Delta$.
Since $S^{-1}_{w_{i}}$ is a bi-Lipschitz map we have $\dim_{A}Z_{i}\leq s$ and so
by definition there exists a constant $C_{i}(s^{*})>0$ such that $\sup_{x\in Z_{i}}N_{r}(B(x,R)\cap Z_{i})\leq C_{i}(s^{*})(R/r)^{s^{*}}$ for all $0<r<R<\infty$ and $s<s^{*}$. Specifically for $s^{*}$ satisfying $s<s^{*}<t$ there exists uniform constant $C$ such that $\sup_{x\in Z_{i}}N_{r}(B(x,R)\cap Z_{i})\leq C(R/r)^{s^{*}}$ and in particular that 
$N_{r}(Z_{i})\leq C^{*}r^{-s^{*}}$
for some $0<C^{*}<\infty$ not depending on $i$, by choosing $R>\lvert\Delta\vert$.
Additionally, it is easy to see that, for some $k_{s}>0$ independent of $i$ and $\epsilon$ (cf.~(\ref{generalSubMult}) and preceding paragraphs) and some $k>0$ related to the difference in length due to the connecting word,
\[
N_{\epsilon^{i}}(Z_{i})\geq kk_{s}^{i} \vertiii{\mathds{1}\mathbf{P}_{\epsilon}^{0}(\xi_{i})\dots\mathbf{P}^{0}_{\epsilon}(\sigma^{i-1}\xi_{i})}.
\]
Thus there exists $C^{**}$ such that
\[
k_{s}^{i} \vertiii{\mathds{1}\mathbf{P}_{\epsilon}^{0}(\xi_{i})\dots\mathbf{P}^{0}_{\epsilon}(\sigma^{i-1}\xi_{i})}\leq C^{**}\epsilon^{-is^{*}},
\]
so
\begin{align*}
s^{*}&\geq
\frac{\log \left[(1/C^{**})k_{s}^{i}\vertiii{\mathds{1}\mathbf{P}_{\epsilon}^{0}(\xi_{i})\dots\mathbf{P}^{0}_{\epsilon}(\sigma^{i-1}\xi_{i})}\right]}{-i\log\epsilon}\\
&\geq
\frac{\log \left((1/C^{**})^{1/i}k_{s}\vertiii{\mathds{1}\mathbf{P}_{\epsilon}^{0}(\xi_{i})\dots\mathbf{P}^{0}_{\epsilon}(\sigma^{i-1}\xi_{i})}^{1/i}\right)}{-\log\epsilon}
\end{align*}
for all $i$.
However the term on the right converges to $t-\log(k_{s})/\log(\epsilon)$ as $i\to\infty$. Since $\epsilon$ was arbitrary, letting $\epsilon\to0$ we have the required contradiction that $t\leq s^{*}<t$.

\vspace{0.8cm}
To prove the upper bound note that since we are assuming the USSC,  the $\epsilon$ approximation sets $K_{v,\epsilon}(\omega)$ are all equal to the attractor $K_{v}(\omega)$ by Lemma~\ref{approxAreSubsets}.
We first show that for any $z\in\R^{d}$ the number of sets of diameter comparable to $\epsilon>0$ intersecting the ball $B(z,\epsilon)$ is uniformly bounded. Let $\Xi^{*}=\{x_{i}\}$ be the set of words in $\mathds{1}_{\epsilon_{0}}\mathbf{H}^{\epsilon}(\omega)$ whose image $S_{x_{i}}(\Delta)$ intersects $B(z,\epsilon)$. Let $c_{\min}>0$ be the least contraction rate. We have
\[
\lvert\Xi^{*}\rvert (\epsilon c_{\min})^{d}=\sum_{x\in\Xi^{*}}(\epsilon c_{\min})^{d}\leq \sum_{x\in\Xi^{*}}\lvert S_{x}(\Delta)\rvert^{d}\leq \lvert B(z,2\epsilon)\rvert^{d}\leq (4\epsilon)^{d};
\]
thus $\lvert\Xi^{*}\rvert\leq (4/c_{\min})^{d}$ is bounded. 

Now let $r$ be such that $0<r<\epsilon$ and define $k_{r}$ to be the unique integer such that $\epsilon^{k_{r}+1}< r\leq \epsilon^{k_{r}}$. For each $x\in\Xi^{*}$ the number of $r$-balls needed to cover $S_{x}(\Delta)\cap K_{v}(\omega)$ is however
bounded by $\sum_{i=1}^{n}(\lVert\mathds{1}\mathbf{P}^{0}_{\epsilon}(\xi_{k_{r}})\dots \mathbf{P}^{0}_{\epsilon}(\sigma^{k_{r}}\xi_{k_{r}})\rVert_{s})_{v,i}$, the maximal way of covering the cylinder with cylinders of diameter $\epsilon^{k_{r}+1}$ or less. Hence
\begin{align}
\sup_{x\in K_{v}(\omega)}N_{r}(B(x,\epsilon)\cap K_{v}(\omega))&\leq \lvert\Xi^{*}\rvert\sum_{i=1}^{n}(\lVert\mathds{1}\mathbf{P}^{0}_{\epsilon}(\xi_{k_{r}})\dots \mathbf{P}^{0}_{\epsilon}(\sigma^{k_{r}}\xi_{k_{r}})\rVert_{s})_{v,i}\nonumber\\
&\leq   \lvert\Xi^{*}\rvert \vertiii{\mathds{1}\mathbf{P}^{0}_{\epsilon}(\xi_{k_{r}})\dots \mathbf{P}^{0}_{\epsilon}(\sigma^{k_{r}}\xi_{k_{r}})}\nonumber\\
&\leq C (\epsilon^{k_{r}+1})^{-(s+\delta)}\leq C \left(\frac{\epsilon}{r}\right)^{s+\delta}\nonumber
\end{align}
for some constant $C>0$ for each $\delta>0$ giving the required upper bound to the Assouad dimension.

\subsection{Proof of Theorem~\ref{infiniteSameDimension}}
Although we will not prove it here, there exists a nice expression for the Hausdorff dimension of the random attractor.
\begin{lma}\label{infiniteHausDim}
	Assume $\mathbf{\Gamma}$ satisfies the USSC, then almost surely, conditioned on $F_{v}$ being non-empty, $\dim_{H}F_{v}=s_{h}$, where $s_{h}$ is the unique non-negative real satisfying
	\begin{equation}\label{inftyHD}
	\rho\left[\E\left( 
	\mathfrak{R}^{s_{h}}
	\begin{pmatrix}
	\bigsqcup_{e\in\tensor*[_{1}]{E}{_{1}}(\omega_1)}e	&	\dots		&	\bigsqcup_{e\in\tensor*[_{1}]{E}{_{n}}(\omega_1)}e \\
	\vdots		&		\ddots		&\vdots		\\
	\bigsqcup_{e\in\tensor*[_{n}]{E}{_{1}}(\omega_1)}e	&	\dots		&	\bigsqcup_{e\in\tensor*[_{n}]{E}{_{n}}(\omega_1)}e
	\end{pmatrix}
	\right)\right]=1.
	\end{equation}
	Here $\rho$ refers to the spectral radius of a matrix.
\end{lma}

Briefly, this can be shown by rewriting the Hutchinson-Moran sum of the $k$th level as a martingale and a proof strategy almost identical to that of Theorem 15.1 in Falconer~\cite{FractalGeo}. Compare also with the results in the introduction of Olsen~\cite{Olsen94}.

Let $\mathbf{K}^{\epsilon}(q)$ be the matrix of words that corresponds to the graph  $\Gamma^{\epsilon}(q)\in{\Gamma}^{\epsilon}$.
Since by Lemma~\ref{sameInfiniteProcess} the attractor $F_{v}^{\epsilon}$ of the approximation is again an $\infty$-variable RGDS which furthermore satisfies the USSC, we can apply Theorem~\ref{infiniteHausDim} and get that $\dim_{H}F_{v}^{\epsilon}=s_{h,\epsilon}$, where
\[
\rho\left[\E_{q\in\mathcal{Q}}\left(\mathfrak{R}^{s_{h,\epsilon}}\mathbf{K}^{\epsilon}(q)\right)\right]=
\lim_{k\to\infty}\lVert\left[\E\mathfrak{R}^{s_{h,\epsilon}}(\mathbf{K}^{\epsilon}(q))\right]^{k}\rVert^{1/k}=1.
\]
The second equality holds by Gelfand's Theorem for any suitable matrix norm, such as $\vertiii{.}_{\sup}$, see for example Arveson~\cite[Theorem 1.7.3]{ArvesonShortCourseSpectralTheory}.
It can be shown that this expectation is a decreasing, continuous function in $s_{h,\epsilon}$ and there is a unique value such that the expectation is equal to $1$. The proof is almost identical to that of Lemma~\ref{almostSurePressureProperties} and we will omit it here. Now as $F_{v}^{\epsilon}\subseteq F_{v}$ we have that $s_{h,\epsilon}\leq s_{h}$, where $s_{h}=\dim_{H}F_{v}$.
We therefore conclude that 
\[
\lim_{k\to\infty}\lVert \left[\E\mathfrak{R}^{s_{h}}(\mathbf{K}^{\epsilon}(q))\right]^{k} \rVert^{1/k}\leq 1.
\]
By an argument similar to that of Theorem~\ref{dimensionEquality}, noting that the diameters of the images are comparable to $\epsilon$, we get
\[
\lim_{k\to\infty}\epsilon^{s_{h}}\lVert \left[\E\mathfrak{R}^{0}(\mathbf{K}^{\epsilon}(q))\right]^{k} \rVert^{1/k}=\epsilon^{s_{h}}\rho\E(\mathfrak{R}^{0}(\mathbf{K}^{\epsilon}(q)))\leq 1,
\]
and as $N_{\epsilon}(F_{v})\asymp \sum_{u\in V}(\mathfrak{R}^{0}(\mathbf{K}^{\epsilon}))_{v,u}$ we have $\E N_{\epsilon}(F_{v})\leq C \epsilon^{-s_{h}}$.
Let $\zeta,\theta>0$ and consider
\[
\sum_{\substack{\delta=\zeta^{k}\\k\in\N}}\Prob\{N_{\delta}(F_{v})\geq \delta^{-(s_{h}+\theta)}\}
\leq\sum_{\substack{\delta=\zeta^{k}\\k\in\N}}\frac{\E N_{\delta}(F_{v})}{\delta^{-(s_{h}+\theta)}}
\leq C\sum_{\substack{\delta=\zeta^{k}\\k\in\N}}\frac{\delta^{-s_{h}}}{\delta^{-s_{h}}\delta^{-\theta}}
\leq C\sum_{k\in\N}\zeta^{k\theta}<\infty.
\]
Now noting that for all $k$ we have $N_{\zeta^{k}}(F_{v})\asymp N_{\zeta^{k+1}}(F_{v})$ so by the Borel-Cantelli Lemma with probability $0$ the event $N_{\delta}(F_{v})\geq \delta^{-(s_{h}+\theta)}$ happens infinitely often and therefore, almost surely,
\[
\limsup_{\delta\to0}\frac{\log N_{\delta}(F_{v})}{-\log\delta}\leq\limsup_{\delta\to 0}\frac{\log\delta^{-(s_{h}+\theta)}}{-\log\delta}=s_{h}+\theta.
\]
But $\theta>0$ was arbitrary, so almost surely $\dim_{B}F_{v}=\dim_{H}F_{v}$, as required.\qed

\subsection{Proof of Theorem~\ref{infiniteAssouadDimension}}
The proof of Theorem~\ref{infiniteAssouadDimension} is very similar to that of Theorem~\ref{AssouadTheo} and we  only highlight the differences and sketch the rest of the proof.
Let $\overline{\mathbf{K}}^{\epsilon}=\mathbf{K}^{\epsilon}(q_{\max})$, where $q_{\max}$ is such that, 
\[
\vertiii{\mathfrak{R}^{0}\mathbf{K}^{\epsilon}(q_{\max})}_{\sup}=\max_{q\in\mathcal{Q}}\vertiii{\mathfrak{R}^{0}\mathbf{K}^{\epsilon}(q)}_{\sup}.
\] 
Furthermore let $\mathcal{R}^{\epsilon}$ be the arrangements of words in the row of $\overline{\mathbf{K}}^{\epsilon}$ that is maximal with respect to the row norm.
Given any finite word $w$ we can therefore construct a maximal $k$-subtree by appending the letters from $\mathcal{R}^{\epsilon}$ to $w$, if necessary by connecting them with a connecting word which is bounded in length $l$. Therefore we can construct a subtree of level $k+l$ such that, for some uniform constant $C>0$, 
\[
N_{\epsilon^{k}}(S^{-1}_{w}(\Delta))\geq C\vertiii{\mathfrak{R}^{0}\underbrace{\overline{\mathbf{K}}^{\epsilon}\dots\,\overline{\mathbf{K}}^{\epsilon}}_{\text{$k$ times}}}_{\sup}.
\]
Noticing that by Gelfand's theorem, 
\[
\vertiii{\mathfrak{R}^{0}\underbrace{\overline{\mathbf{K}}^{\epsilon}\dots\,\overline{\mathbf{K}}^{\epsilon}}_{\text{$k$ times}}}_{\sup}^{1/k}
\to \rho(\mathfrak{R}^{0}\overline{\mathbf{K}}^{\epsilon})\;\;\;\text{ as }\;\;\;k\to\infty,
\]
and that for every $k$ we can find a sequence of words $\{w_{i}\}$ that has this maximal $i+l$ subtree splitting for almost every realisation $q\in\mathcal{Q}$, we can apply the same argument as in Theorem~\ref{AssouadTheo} to conclude that almost surely
\[
\dim_{A}F_{v}\geq\sup_{\epsilon>0}\frac{\log\rho(\mathfrak{R}^{0}\overline{\mathbf{K}}^{\epsilon})}{-\log \epsilon}.
\]
Assuming the USSC the upper bound follows immediately as $\rho(\mathfrak{R}^{0}\overline{\mathbf{K}}^{\epsilon})$ is by definition the largest eigenvalue and hence greatest rate of expansion. The argument is identical to Theorem~\ref{AssouadTheo} and is left to the reader.
\qed

\end{document}